\documentclass[10pt]{amsart}

\usepackage{latexsym}
\usepackage{amssymb}
\usepackage{amsmath}

\newtheorem{theorem}{Theorem}[section]
\newtheorem{lemma}[theorem]{Lemma}

\numberwithin{equation}{section}

\newtheorem{Theorem}{Theorem}[section]
\newtheorem{thm}[Theorem]{Theorem}

\newtheorem{proposition}[Theorem]{Proposition}
\newtheorem{prop}[Theorem]{Proposition}

\newtheorem{cor}[Theorem]{Corollary}

\newtheorem{claim}{Claim}

\theoremstyle{definition}

\newtheorem{definition}[Theorem]{Definition}

\newcommand{\Imp}{\mbox{$\Longrightarrow$}}
\newcommand{\Iff}{\mbox{$\Longleftrightarrow$}}

\def\proof{{\bf Proof.}\ }

\def\ie{\emph{i.e.}}
\def\pes{\emph{e.g.}}
\def\Pes{\emph{E.g.}}

\def\F{{\mathcal F}}

\def\L{\mathcal{L}}

\def\P{{\mathcal P}}

\def\RR{{\mathcal{R}}}
\def\RP{{\mathcal{R}^{+}}}
\def\U{{\mathcal U}}

\def\V{{\mathcal V}}

\def\QQ{{\mathcal Q}}

\def\EE{{\mathbb E}}
\def\N{{\mathbb N}}
\def\Q{{\mathbb Q}}

\def\Z{{\mathbb Z}}

\def\SSU{{\mathbb{S}}_{\U}}
\def\TTU{{\mathbb{T}}_{\U}}
\def\Pp{{\mathbb{P}}}

\def\WW{{\mathbf{F}}}

\def\TT{{\mathbf{t}}}

\def\II{{\mathbb{I}}}

\def\f{\widetilde{f}}

	\def\B{\mathcal{B}}
\def\fu{\fg_{\,\U}}

\def\ag{{\alpha}}
\def\dg{{\delta}}

\def\eg{{\varepsilon}}
\def\fg{{\varphi}}

\def\og{{\omega}}

\def\sg{{\sigma}}
\def\Sg{{\Sigma}}

\def\cov{\mbox{\bf cov}}

\def\+#1{\vec{#1}}

\def\eb{{\mathbf{e}}}

\def\ck{{\mathfrak{c}}}

\def\xk{{\mathfrak{x}}}
\def\yk{{\mathfrak{y}}}
\def\zk{{\mathfrak{z}}}
\def\wk{{\mathfrak{w}}}
\def\nk{{\mathfrak{n}}}

\def\nku{{\mathfrak{n}}_{\U}}
\def\Nku{{\mathfrak{N}}_{\U}}
\def\dk{{\mathfrak{d}}}

\def\Ik{{\mathfrak{I}}}
\def\Kk{{\mathfrak{K}}}
\def\Sk{{\mathfrak{S}}}

\def\ng{{\mathfrak{n}}}

\def\Ng{{\mathfrak{N}}}
\def\Nk{{\mathfrak{N}}}
\def\Rk{{\mathfrak{R}}}
\def\Rg{{\mathfrak{R}}}

\def\*{\times}

\def\0{\emptyset}
\def\/{\setminus}
\def\_{\overline}
\def\eq{\approx}
\def\<{\prec}
\def\Equ{\equiv_{\U}}

\def\ult#1#2{^{#1}_{\; #2}}

\def\incl{\subseteq}

\def\pincl{\subset}

\def\la{\langle}
\def\ra{\rangle}

\def\equ{\approx_{\U}}

\def\zfc{\textsf{ZFC}}

\def\CH{\textsf{CH}}

\def\qed{${}$\hfill $\Box$}
\def\dim{\textbf{Proof.}}
\def\SZ{\Z\la\!\la \TT\ra\!\ra}
\def\Sud{\Z\la\!\la t_{0},t_{1}\ra\!\ra}
\def\etc{\textit{etc.}}

\def\fu{F_{\U}}
\def\euno{\textsf{(E1)}}
\def\edue{\textsf{(E2)}}
\def\etre{\textsf{(E3)}}

\def\ecinque{\textsf{(E4)}}
\def\easy{\textsf{(E0)}}

\newenvironment{lsnum}{\begin{enumerate}}{\end{enumerate}}
\newenvironment{pf}{\begin{proof}}{\end{proof}}

\newcommand{\scr}[1]{\ensuremath{\mathcal {#1}}}

\newcommand{\emp}{\varnothing}
\renewcommand{\phi}{\varphi}

\newcommand{\sq}[1]{\ensuremath{\langle#1\rangle}}

\newcommand{\restr}{\mathop{\upharpoonright}}

\newcommand{\notarrow}{\kern .42em\not\kern -.42em\longrightarrow}

\begin{document}

\title[Quasi-selective ultrafilters and asymptotic numerosities]
{Quasi-selective ultrafilters
\\
and asymptotic numerosities}

 \author{Andreas Blass}
 \address{Mathematics Department, University of Michigan at Ann Arbor, USA. 
\tt{ablass@umich.edu}}
 \thanks{A. Blass has been partially supported by NSF grant DMS-0653696.}

\author{Mauro Di Nasso}
\address{Dipart. di Matematica ``L. Tonelli'',
Universit\`{a} di Pisa, Italy.
\tt{dinasso@dm.unipi.it}} 
\thanks{M. Di Nasso and M. Forti have been supported by MIUR grant 
PRIN2007.} 

\author{Marco Forti}
\address{Dipart. di Matematica Applicata ``U. Dini'',
Universit\`{a} di Pisa, Italy.
\tt{forti@dma.unipi.it}} 


\begin{abstract}
We isolate a new class of ultrafilters on $\N$, called 
``quasi-selective''  because they are
 intermediate between selective ultrafilters and $P$-points.
(Under the Continuum Hypothesis these three classes are 
distinct.)
The existence of quasi-selective ultrafilters is equivalent
to the existence of ``asymptotic numerosities'' for all sets of tuples
$A\subseteq\N^k$.  Such numerosities are hypernatural numbers that 
generalize finite cardinalities 
to countable point sets. Most notably, they maintain the structure of ordered 
semiring, and, in a precise sense, they allow for a natural extension of 
 asymptotic density to all sequences of tuples of natural numbers.
\end{abstract}

\maketitle

\section*{Introduction}

Special classes of ultrafilters over $\N$ have been
introduced and variously applied in the literature,
starting from the pioneering work by G. Choquet \cite{co1,co2}
in the sixties.

In this paper we introduce a new class of ultrafilters,
namely the \emph{quasi-selective ultrafilters}, as a tool to generate
a good notion of ``equinumerosity'' on the sets of tuples
of natural numbers (or, more generally, on all point sets
$A\subseteq\L^k$ over a countable line $\L$).
By equinumerosity we mean an equivalence relation
that preserves the basic properties of
equipotency for finite sets, including the Euclidean principle that 
``the whole is greater than the part''. More precisely,
we require that -- similarly as finite numbers --
the corresponding numerosities  be the non-negative part
of a discretely ordered ring, where $0$ is the numerosity of the 
empty set,
$1$ is the numerosity of every singleton,
and sums and products of numerosities correspond
to disjoint unions and Cartesian products, respectively.

This idea of numerosities that generalize finite cardinalities
has been recently investigated by V.~Benci, M.~Di Nasso and 
M.~Forti
in a series of papers, starting from \cite{bd}, where a numerosity is 
assigned to each pair $\la A,\ell_{A}\ra$, depending on the (finite-to-one)
``labelling 
function'' $\ell_{A}:A\to\N$. 
The existence of a numerosity function for labelled sets
turns out to be equivalent to the existence of a 
selective ultrafilter.
That research was then continued by investigating a similar 
notion of numerosity  for
sets of arbitrary cardinality, namely: sets of ordinals in \cite{bdf}, 
subsets of a superstructure in \cite{logan}, point sets over the real line 
in \cite{df}. 
A related notion of ``fine density'' for sets of natural 
numbers is introduced and investigated in  \cite{dn}. 
In each of these contexts special classes of 
ultrafilters over large sets naturally arise.

Here we focus on subsets $A\subseteq\N^k$ of tuples
of natural numbers, and we show that the existence
of particularly well-behaved equinumerosity
relations (which we call "asymptotic") for such sets
is equivalent to the existence of another special kind
of ultrafilters, named quasi-selective ultrafilters.
Such ultrafilters may be
of independent interest, because they are closely
related (but not equivalent) to other well-known classes
of ultrafilters that have been extensively considered in the
literature. In fact, on the one hand, all selective 
ultrafilters
are quasi-selective and all quasi-selective ultrafilters are P-points.
 On the other hand, it is consistent that these three classes of 
 ultrafilters are distinct.

\smallskip
The paper is organized as follows. In Section \ref{qsu} we introduce 
the class of quasi-selective ultrafilters on $\N$ and we study their 
properties, in particular their relationships with $P$-points and 
selective ultrafilters. In Section \ref{qsconst}, assuming the 
Continuum Hypothesis, we present a general 
construction of quasi-selective non-selective ultrafilters, that are 
also weakly Ramsey in the sense of \cite{bl74,qswr}.
In Section \ref{equi} we introduce axiomatically a general notion 
of ``equinumerosity'' for sets of tuples of natural numbers. 
 In Section \ref{arit} 
we show that the resulting numerosities, where 
 sum, product and ordering are defined in the standard Cantorian way,
 are the non-negative part of an ordered 
ring. Namely this ring is isomorphic to the quotient of a ring of 
power-series modulo a suitable 
ideal.
In Section \ref{asy} we introduce the special notion of 
``asymptotic'' equinumerosity, which generalizes the fine density of 
\cite{dn}. We show  that there is a one-to-one correspondence between 
asymptotic equinumerosities and quasi-selective ultrafilters, where 
equinumerosity is witnessed by a special class of bijections 
depending on the ultrafilter
(``$\U$-congruences''). 
The corresponding semiring of 
numerosities  is isomorphic to an initial cut of the ultrapower 
$\N\ult{\N}{\U}$. In particular, asymptotic numerosities exist if and 
only if there exist quasi-selective ultrafilters.
Final remarks and open questions are contained in the concluding 
Section \ref{froq}.

\smallskip
In general, we refer to \cite{CK} for definitions and basic facts
concerning ultrafilters, ultrapowers, and nonstandard models, and to 
\cite{bo} for special ultrafilters over $\N$.

\smallskip
The authors are grateful to  Vieri Benci for many useful discussions
and suggestions.

\section{quasi-selective  ultrafilters}\label{qsu}

We denote by $\N$ the set of all \emph{nonnegative integers}, and by 
$\N^{+}$ the subset of all \emph{positive} integers.

Recall that, if $\F$ is a filter on $X$, then two functions $f,g:X\to Y$ 
are called \emph{$\F$-equivalent} if $\{x\in X \mid f(x)=g(x) \}\in\F$. In this 
case we write $f\equiv_{\F}g$.
\begin{definition}\label{qsel}
{A nonprincipal ultrafilter $\U$ on $\N$ is called \emph{quasi-selective}
if every function $f$ such that $f(n)\le n$ for all $n\in\N$ is $\U$-equivalent
to a nondecreasing one.
}
\end{definition}

The name `quasi-selective' recalls one of the characterizations of 
\emph{selective} (or \emph{Ramsey}) ultrafilters (see 
\pes~\cite[Prop.~4.1]{bd}), namely
\begin{itemize}
    \item   \emph{The ultrafilter $\U$ on $\N$ is selective if and only if 
every $f:\N\to\N$ is $\U$-equivalent to a nondecreasing function. } 
\end{itemize}
In particular all selective ultrafilters are quasi-selective.

Let us call  ``interval-to-one'' a function $g:\N\to\N$  such that, 
for all $n$,
$g^{-1}(n)$ is a (possibly infinite) interval of $\N$. 

\begin{prop}\label{int1}
    Let $\U$ be a quasi-selective ultrafilter. Then any partition of 
    $\N$ is an \emph{interval partition} when restricted to a suitable set 
    in $\U$. Hence every 
    $f:\N\to\N$ is 
$\U$-equivalent to an ``interval-to-one'' function.
 In particular all quasi-selective ultrafilters are P-points.

 \end{prop}   

    \proof
    Given a 
    partition of $\N$, consider the function $f$ mapping each number 
    to the least element of its class. Then $f(n)\le n$ for all 
    $n\in\N$, so, by quasi-selectivity, there exists a set $U\in\U$ 
    such that the restriction
    $f_{\upharpoonleft U}$ is nondecreasing. Then  the given partition is 
    an interval partition when restricted to $U$.
    \qed

    \medskip
    We say that a function $f:\N\to\N$ has 
\emph{polynomial growth}
if it is eventually dominated by some polynomial, 
\ie~if
there exist $k,m$ such that for all $n>m$, $f(n)\le n^k$.

We say that a function $f:\N\to\N$ has \emph{minimal steps} if 
$|f(n+1)-f(n)|\le 1$ for all $n\in\N$.

\begin{proposition}\label{qseleq}
The following properties are equivalent for a nonprincipal 
ultrafilter $\U$ on $\N$:
\begin{enumerate}
    \item  $\U$ is quasi-selective;
    
    \item every function of polynomial growth is $\U$-equivalent
to a nondecreasing one;

    \item every function with minimal steps is $\U$-equivalent
to a nondecreasing one.\footnote{~
Ultrafilters satisfying this property are called \emph{smooth} in \cite{dn}.}

\end{enumerate}

\end{proposition}

\proof
$(1)\Imp (2)$.   We prove that if every function $f<n^{k}$ is 
$\U$-equivalent to a nondecreasing one, 
then the same property holds for every function $g<n^{2k}$. 
The thesis then follows by induction on $k$.
Given $g$, let $f$ be the integral part of the square root of $g$. So 
$g<f^{2}+2f+1$, and hence $g=f^{2}+f_{1}+f_{2}$ for suitable functions
$f_{1},f_{2}\le f<n^{k}$. By hypothesis we can pick nondecreasing 
functions $f',f'_{1},f'_{2}$ that are $\U$-equivalent to 
$f,f_{1},f_{2}$, respectively.
Then clearly $g$ is $\U$-equivalent to the nondecreasing function
$f'^{2}+f'_{1}+f'_{2}$.

$(2)\Imp (3)$ is trivial.

$(3)\Imp (1)$. We begin by showing that $(3)$ implies the following
\begin{claim} There exists $U=\{u_{1},\ldots,u_{n}\ldots\}\in\U$ such 
    that $u_{n+1}>2u_{n}$.  
\end{claim}
        
Define three minimal step functions 
$f_{0},f_{1},f_{2}$ as follows, for $k\in\N^{+}$: 
$$f_{0}(m)= 2^{k-1}-|3\cdot 2^{k-1}-m|\ \ \mbox{for}\ \ 2^{k}\le m\le 
2^{k+1}$$
$$f_{1}(m)= 9\cdot 2^{2k-1}-|15\cdot 2^{2k-1}-m|\ \ \mbox{for}\ \ 
3\cdot 2^{2k}\le m\le 
3\cdot 2^{2k+2} \ $$
$$f_{2}(m)= 9\cdot 2^{2k-2}-|15\cdot 2^{2k-2}-m|\ \ \mbox{for}\ \ 
3\cdot 2^{2k-1}\le m\le 
3\cdot 2^{2k+1}$$
The graphs of these functions are made up of the catheti of 
isosceles  right
triangles whose hypotenuses are placed on the horizontal axis. 
The 
function $f_{0}$ is decreasing in the intervals $[3\cdot 
2^{k-1},2^{k+1}]$, whereas in the intervals 
$[2^{k},3\cdot 
2^{k-1}]$ the function $ f_{1}$ is decreasing for odd $k$, and 
$f_{2}$ is decreasing for even $k$. The ultrafilter $\U$ contains a set 
$V$ on which 
all three of these functions are nondecreasing. Such 
 a set $V$ has at most one point in each interval. 
Starting from each one of the first four 
points of $V$, partition $V$ into four parts by taking every fourth point, so 
as to obtain four sets satisfying the condition of the claim. Then exactly
one of the resulting sets belongs to $\U$, and the claim follows. 

Now remark that every 
    function $f\le n$ can be written as a sum $f_{1}+f_{2}$, where 
    $f_{1},f_{2}\le \lceil\frac{n}{2}\rceil$. 
Pick a set 
$U\in\U$ as given by the claim. Then both functions $f_{1}$ and 
$f_{2}$ agree on $U$
    with  suitable minimal step functions, because 
    $u_{n+1}-u_{n}> \frac{u_{n+1}}{2}$, whereas
$g\le \lceil\frac{n}{2}\rceil$ implies
    $|g(u_{n+1})-g(u_{n})|\le \lceil\frac{u_{n+1}}{2}\rceil$.  So $f_{1},f_{2}$ 
    are equivalent  modulo $\U$ to two nondecreasing functions 
    $f'_{1},f'_{2}$ respectively, and $f$ is equivalent  modulo $\U$ to their sum
    $f'_{1}+f'_{2}$, which is nondecreasing as the sum of nondecreasing 
    functions.
           \qed

\begin{thm}\label{FU}
    Let $\U$ be a quasi-selective ultrafilter, and let $f:\N\to\N$ be 
    non\-decreasing. Then the following properties are equivalent:
    \begin{enumerate}
        \item[$(i)$]  for every function $g\le f$ there exists a
	nondecreasing function $h\Equ g$;
    
        \item[$(ii)$]   there exists $U=\{u_{n}\mid n\in\N\}\in\U$ such that 
	$f(u_{n})< u_{n+1}-u_{n}$.
    \end{enumerate}
    
\end{thm}

\proof
$(i)\Imp (ii)$. Define inductively 
the sequence $\la x_{n}\mid n\in\N\ra$ by putting $x_{0}=1$ and
 $x_{n+1}=f(x_{n})+x_{n}$. Define 
$g:\N\to\N$ by $g(x_{n}+h)=f(x_{n})-h$ for $0\le h < f(x_{n})$.
Assuming $(i)$, there is a set in $\U$ which meets each interval 
$[x_{n},x_{n+1})$ in one point $a_{n}$. So by putting either 
$u_{n}=a_{2n}$ or $u_{n}=a_{2n+1}$ we obtain a set $U$ satisfying the 
condition $(ii)$. Namely, in the even case we have
$$u_{n+1}-u_{n}>x_{2n+2}-x_{2n+1}=f(x_{2n+1}) \ge f(u_{n}),$$
and similarly in the odd case.

$(ii)\Imp (i)$.  By $(ii)$ we may pick $U\in\U$ such that 
 $x<y$ both in $U$ 
 implies $y>f(x)$. Given $g\le f$, partition $U$ as follows
 $$\begin{array}{ccl}
     U_{1} & = & \{u\in U\mid \forall x\in U \,(x<u\ \Imp\ g(x)\le 
     g(u)\,)\}  \\
     U_{2} & = & \{u\in U\mid \exists x\in U \,(x<u\ \&\ g(x)> 
 g(u)\,)\}
 \end{array}$$
   Then $g$ is nondecreasing on $U_{1}$, so 
we are done  when $U_{1}$ belongs to $\U$. Otherwise $U_{2}\in\U$ and
we have $g(u)<u$ for all  
$u\in U_{2}$. In fact, given 
$u\in U_{2}$,  pick $x\in U$ such that $x<u$ and $g(x)> g(u)$: then
$$g(u) <g(x)\le f(x)< u.$$
 Then $(i)$ follows by quasi-selectivity of $\U$.
\nopagebreak
\qed

\bigskip

Let us denote by $\fu$ the class of all  functions $f$ satisfying the 
equivalent conditions of the 
above theorem
$$\fu=\{ f:\N\to\N\ \mbox{\rm{nondecreasing}}\,\mid\, g\le f\ \Imp\ 
\exists h\ \mbox{\rm{nondecreasing s.t.}}\ h\Equ g\, \}.$$
Recall that, if $\,\U$ is quasi-selective, then every function is 
$\U$-equivalent to an ``interval-to-one'' function.  As a consequence,
the class $\fu$ ``measures the selectivity'' of quasi-selective 
ultrafilters, according to the following proposition.

\begin{proposition}\label{missel}

Let $\U$ be a quasi-selective ultrafilter and let $g$ be  an unbounded 
interval-to-one function. Define the function $g^{+}$ by 
$$g^{+}(n)=\max\,\{x\mid g(x)=g(n)\},$$ and let $e_{g}$ be the 
function enumerating  the range of $g^{+}$.
Then the following are equivalent:
\begin{enumerate}
    \item $g$ is 
$\U$-equivalent 
to a one-to-one function; 

    \item  $g^{+}$ belongs to $\fu$;

    \item  there exists 
 $U=\{u_{0}<u_{1}<\ldots<u_{n}<\ldots\}\in\U$ such that $u_{n}>e_{g}(n)$.
\end{enumerate}
 
\end{proposition}

\proof

(1)$\Imp$(2).  
Let $g$ be one-to-one on $U\in\U$. Then each 
interval where $g$ is constant contains at most one point of $U$. 
Hence  $g^{+}$ 
is increasing on $U$. Let $h\le g^{+}$ be given, and put 
$$ U_{1}  =  \{u\in U\mid h(u)\le 
     u\,\},\ \ 
     U_{2}  = \{u\in U\mid h(u)> 
     u\,\}.$$
     
     Then $u<h(u)\le g^{+}(u)$ for $u\in U_{2}$, and hence $h$ is 
     increasing when restricted to $U_{2}$. So we are done when $U_{2}\in\U$.
     On the other hand, when $U_{1}\in\U$, the function $h$ is 
     $\U$-equivalent to a nondecreasing one, by 
     quasi-selectivity.
     
     \medskip
     (2)$\Imp$(3). By Theorem \ref{FU}, there exists a set
      $U=\{u_{n}\mid n\in\N\}\in\U$ such that 
     $g^{+}(u_{n})< u_{n+1}-u_{n}$.  Suppose that $g(u_{n})=g(u_{n+1})$ 
     for some $n$; then 
     $g^{+}(u_{n})\ge u_{n+1}\ge u_{n+1}-u_{n} > g^{+}(u_{n})$, a
     contradiction. Hence $g$ is one-to-one on $U$, and so before
     $u_{n+1}$ there are at least $n$ intervals of the form 
     $g^{-1}(k)$. Therefore 
     $u_{n+1}> e_{g}(n)$, and $U\/ \{u_{0}\}$ satisfies 
     condition (3).    
     
     \medskip
     (3)$\Imp$(1).
     For each $n\in\N$ let $k$ be the unique number such that $u_{n}$ 
     lies in the interval $[e_{g}(n+k-1), e_{g}(n+k))$. Then let $h$ 
     be the unique number such that $e_{g}(n+k)$ lies in the interval 
     $(u_{n+h-1}, u_{n+h}]$. Thus we have
     $$e_{g}(n+k-1)\le u_{n}\le 
 u_{n+h-1}<e_{g}(n+k)\le u_{n+h}.$$ 
 Then $k \ge h$ and $u_{n}\ge k$, and we can define the function $f$ on 
 $U$ by $f( u_{n+i})=k-i$, for $0\le i<h$. Since $f(u)\le u$ for all $u\in U$ there 
exists a set $V\in \U$ on which $f$ is nondecreasing. Then $g$ is 
one-to-one on $V\cap U$.
     
\qed

\medskip
Recall that the ultrafilter $\U$  is \emph{rapid} if for
 every increasing function $f$ there exists 
 $U=\{u_{0}<u_{1}<\ldots<u_{n}<\ldots\}\in\U$ 
such that $u_{n}>f(n)$.
 The equivalence of the conditions in the above corollary yields

\begin{cor}

A quasi-selective ultrafilter is selective if and only if
 it is rapid.

 \qed
\end{cor}

\smallskip

The class of   functions $\fu$ has the following closure 
properties:

\begin{proposition}\label{clos}
Let $\U$ be a quasi-selective  ultrafilter, and let 
$$\fu=\{ f:\N\to\N\ \mbox{\rm{nondecreasing}}\,\mid\, g\le f\ \Imp\ 
\exists h\ \mbox{\rm{nondecreasing s.t.}}\ h\Equ g\, \}.$$ Then
\begin{enumerate}
   \item for all $f\in\fu$  also $\f\in\fu$, where
    $$\f(n)=f^{\circ f(n)}(n)=\underbrace{(f\circ f\circ \ldots 
    \circ f)}_{f(n)\ 
    \mbox{{\scriptsize{\em times}}}}(n).\footnote{
    ~Here we agree that $f^{\circ 0}(n)=n$.}$$  

\item Every  sequence $\la f_{n}\mid n\in\N\ra$ in 
   $\fu$ is dominated by a  function $f_{\og}\in\fu$, \ie\
   for all $n$ there exists $k_{n}$ such that
   $f_{\og}(m)>f_{n}(m)$ for all $m>k_{n}$.   

\noindent
In particular the left cofinality of the  gap determined by
$\fu$ in the ultrapower $\N\ult{\N}{\U}$ 
 is uncountable. 
\end{enumerate}
\end{proposition}

\proof
We prove first that  $\f$ fulfills property $(ii)$ of Theorem 
\ref{FU},
provided $f$ fulfills both properties $(i)$ and $(ii)$ of the same theorem.
By possibly replacing $f$ by $\max \{f, id\}$, we may assume without 
loss of generality that 
$f(n)\ge n$ for all $n\in\N$.

Let $U=\{u_{0}<u_{1}<\ldots<u_{n}\ldots\}\in\U$  be given by property
$(ii)$ for $f$.
Define inductively the sequence $\sg:\N\to\N$ by $$\sg(0)=1 \ 
\mbox{and}\ \sg(n+1)=\sg(n)+f(u_{\sg(n)}).$$
Define the function $g$ on $U$ by $$g(u_{\sg(n)+j})=f(u_{\sg(n)})-j\ \
\mbox{for}\ 0\le j< f(u_{\sg(n)}).$$
Since $g\le f$ on $U$, there exists a subset $V\in\U$ on which $g$ is 
nondecreasing. For each $n$, such a $V$ contains at most one point 
$v_{n}=u_{\tau(n)}$ 
with $\sg(n)\le \tau(n) <\sg(n+1)$. Assume without loss of generality 
that the set 
$\{v_{2n}\mid n\in\N\}\in\U$. 
We shall complete the proof by showing that $v_{2n+2}-v_{2n}> 
\f(v_{2n})$.
Put
$k=\tau(2n+2)-\tau(2n)$; then

$$\begin{array}{ccl}
v_{2n+2}-v_{2n} & = & 
\sum_{i=0}^{k-1}(u_{\tau(2n)+i+1}-u_{\tau(2n)+i}) > 
\sum_{i=0}^{k-1}f(u_{\tau(2n)+i})\ge \\
{} & {} & {}\\
    {} & \ge & f(u_{\tau(2n)+k-1})\ge f(f(u_{\tau(2n)+k-2}))\ge
    \ldots\ge f^{\circ k}(u_{\tau(2n)}).
 \end{array}$$
Now
$$k=\tau(2n+2)-\tau(2n)>\sg(2n+2)-\sg(2n+1)=f(u_{\sg(2n+1)})\ge 
f(v_{2n}).$$
Hence
$$v_{2n+2}-v_{2n}>f^{\circ k}(u_{\tau(2n)})\ge f^{\circ f(v_{2n})}(v_{2n})
=\f(v_{2n}).$$
This completes the proof of $(1)$.

\bigskip
In order to prove point $(2)$, let sets $U_{n}\in\U$ be chosen so as 
to satisfy the property $(ii)$ with respect to the function $f_{n}$. 
As $\U$ 
is a $P$-point, we can take 
$V\in\U$  almost included in every $U_{n}$, \ie\ $V\/ U_{n}$ finite
for all $n\in\N$. Define 
the function $f_{\og}$ by 
$$ f_{\og}(m)= \min\{v'-v\mid v',v\in V,\, v'>v\ge m\}.  $$
Let $k_{n}$ be such that, for all $v\in V$, $v\ge k_{n}$ implies 
$v\in U_{n}$. Given $m>k_{n}$ let $ f_{\og}(m)=v'-v$,
with $m < v <v'$ as required by the definition of $f_\omega$,
and let $u$ be the successor of 
$v$ in $U_{n}$. Then $$ f_{n}(m)\le f_{n}(v)< u-v\le v'-v=f_{\og}(m),$$
and so $f_{\og}$ dominates every $f_{n}$.
Now if $V=\{v_{0}<v_{1}<\ldots\}$, then 
$f_{\og}(v_{k})<v_{k+2}-v_{k}$. So  either $U=\{v_{2n}\mid 
n\in\N\}$ or $U'=\{v_{2n+1}\mid 
n\in\N\}$ witnesses the property $(ii)$ of Theorem \ref{FU} for $f_{\og}$,
and $(2)$ follows.
\nopagebreak\qed

\bigskip
Remark that all \emph{Ackermann functions}\footnote{~
Recall that $f_{m}(n)=A(m,n)$ can be inductively defined by 
$$f_{0}(n)=A(0,n)=n+1, \ \
f_{m+1}(n)=A(m+1,n)=\underbrace{(f_{m}\circ f_{m}\circ \ldots 
    \circ f_{m})}_{n+1\ 
    \mbox{{\scriptsize{\em times}}}}(1).$$} 
$f_{m}(n)=A(m,n)$ belong to $\fu$, because $f_{m+1}\le 
\widetilde{f}_{m}$. Since every primitive recursive function is 
eventually dominated by some $f_{m}$, we
 obtain the following property of ``primitive recursive 
rapidity'':
\begin{cor}
    Let $\U$ be a quasi-selective ultrafilter, and let $f:\N\to\N$ be 
    primitive recursive. Then $f$ is nondecreasing modulo $\U$, and 
    there exists a set
    $U=\{u_{0}<u_{1}<\ldots<u_{n}\ldots\}\in\U$ 
    with\ $~u_{n+1}>f(u_{n})$.
    \qed
\end{cor}

\smallskip

As proved at the beginning of this section, one has the 
following implications

\smallskip
\begin{center}
    \emph{selective\ \ $\Imp$\ \ quasi-selective\ \ $\Imp$\  P-point.}
\end{center}

\smallskip
\noindent  Not even the existence of P-points can be proven in \zfc\ 
(see \pes\ 
\cite{wi,she}), so the 
question as to whether the above three classes of ultrafilters are 
distinct only makes sense under additional hypotheses. 

However the following holds in \zfc:
\begin{proposition}\label{blp}
Assume that the ultrafilter $\U$ is not a Q-point. Then there exists 
an ultrafilter $\U'\cong \U$ that is not quasi-selective.    
\end{proposition}

\dim\
Let \scr U be an ultrafilter that is not a Q-point, so there is a partition 
of $\N$ into finite sets $F_n$ $(n\in\omega)$ such that every set 
in \scr U meets some $F_n$ in more than one point.  Inductively choose 
pairwise disjoint sets $G_n\subseteq\N$ such that, for each $n$,
\[
\min(G_n)>|G_n|=|F_n|.
\]
Let $f:\N\to\bigcup_n G_n$ be such that the restrictions 
$f\restr F_n$ are bijections $f\restr F_n:F_n\to G_n$ for all $n$.  
Since the $G_n$ are pairwise disjoint, $f$ is one-to-one, and therefore 
the ultrafilter $\scr V=f(\scr U)$ is isomorphic to \scr U.  We shall complete 
the proof by showing that \scr V is not quasi-selective.

Define $g:\bigcup_nG_n\to\N$ by requiring that, for each $n$, the 
restriction $g\restr G_n$ is the unique strictly decreasing bijection 
from $G_n$ to the initial segment $[0,|G_n|)$ of $\N$.  Notice that 
the values $g$ takes on $G_n$ are all $<\min(G_n)$; thus $g(x)<x$ for 
all $x\in G_n$.  Extend $g$ to all of $\N$ by setting $g(x)=0$ for 
$x\notin\bigcup_nG_n$.  Now $g:\N\to\N$ and $g(x)\leq x$ for all $x$.  
If \scr V were quasi-selective, there would be a set $A\in\scr V$ on which $g$ 
is non-decreasing.  Since $g$ is strictly decreasing on each $G_n$, each 
intersection $A\cap G_n$ would contain at most one point.  Therefore each 
of the pre-images  $f^{-1}(A)\cap F_n$ would contain at most one point.  
But from $A\in\scr V$, we infer that $f^{-1}(A)$ is in \scr U and therefore
meets some $F_n$ in at least two points.  This contradiction shows that \scr 
V is not quasi-selective.

\qed

Hence the mere existence of a non-selective P-point yields also the 
existence 
of non-quasi-selective P-points. So the second 
implication can be reversed only if the three classes are the same. 
Recall that this possibility has been shown consistent by Shelah 
(see \cite[Section~XVIII.4]{she}).

We conclude this section by stating a theorem that settles the question 
under the Continuum Hypothesis $\CH$:
\begin{thm}
\label{blch} Assume \emph{$\CH$}.
Then there exist $2^\mathfrak{c}$ pairwise non-isomorphic
P-points that are not quasi-selective, and $2^\mathfrak{c}$ 
pairwise non-isomorphic
quasi-selective ultrafilters that are not selective.
\end{thm}

The first assertion of this
theorem follows by combining Proposition \ref{blp} with the known fact
that $\CH$ implies the existence of $2^{\mathfrak c}$ non-isomorphic
non-selective P-points.
The rather technical proof of the second assertion of the theorem is 
contained in the next section. A few  
open questions involving quasi-selective 
ultrafilters are to be found in the final Section \ref{froq}.

\section{A construction of quasi-selective ultrafilters}\label{qsconst}

This section is entirely devoted to the proof of the following 
theorem, which in turn will yield the second assertion of Theorem \ref{blch} above.

\begin{thm} \label{constr}
Assume \emph{\CH}.  For every selective ultrafilter \scr U, there is a non-selective but 
quasi-selective ultrafilter \scr V above \scr U in the Rudin-Keisler 
ordering.\footnote{
~Recall that \scr V is above \scr U in the Rudin-Keisler 
(pre)ordering if there exists a function $f$ such that $\U=f(\V)$.}  
Furthermore, \scr V can be chosen to satisfy the partition relation 
$\N\to[\scr V]^2_3$.
\end{thm}

The square-bracket partition relation in the theorem means that, if 
$[\N]^2$ is partitioned into 3 pieces, then there is a set $H\in\scr V$ such 
that $[H]^2$ meets at most 2 of the pieces.  This easily implies by induction 
that, if $[\N]^2$ is partitioned into any finite number of pieces, then there 
is a set $H\in\scr V$ such that $[H]^2$ meets at most 2 of the pieces.  It is 
also known (\cite{bl74}) to imply that \scr V is a P-point and that
\scr U is, up to isomorphism, the only non-principal ultrafilter
strictly below \scr V in the Rudin-Keisler ordering. 

It will be convenient to record some preliminary information before starting 
the proof of the theorem.  Suppose \scr X is an upward-closed (with respect 
to $\subseteq$) family of finite subsets of $\N$.  Call \scr X \emph{rich} 
if every infinite subset of $\N$ has an initial segment in \scr X.  
(This notion resembles Nash-Williams's notion of a \emph{barrier}, but it is not  
the same.)  Define $\rho\scr X$ to be 
the family of those finite $A\subseteq\N$ such that $A\to(\scr X)^2_2$, i.e., 
every partition of $[A]^2$ into two parts has a homogeneous set in \scr X.  
(The notation $\rho$ stands for ``Ramsey''.)

\begin{lemma}
If \scr X is rich, then so is $\rho\scr X$.
\end{lemma}

\begin{pf}
This is a standard compactness argument, but we present it for the sake of 
completeness.  Let $S$ be any infinite subset of $\N$, and, for each $n\in\N$, 
let $S_n$ be the set of the first $n$ elements of $S$.  We must show that 
$S_n\in\rho\scr X$ for some $n$.  Suppose not.  Then, for each $n$, there 
are  counterexamples, i.e., partitions $F:[S_n]^2\to2$ with no homogeneous 
set in \scr X.  These counterexamples form a tree, in which the predecessors 
of any $F$ are its restrictions to $[S_m]^2$ for smaller $m$.  This tree is 
infinite but finitely branching, so K\"onig's infinity lemma gives us a path 
through it.  The union of all the partitions along this path is a partition 
$G:[S]^2\to2$, and by Ramsey's theorem it has an infinite homogeneous set 
$H\subseteq S$.  Since \scr X is rich, it contains $H\cap S_n$ for some $n$. 
But then one of our counterexamples, namely $G\restr[S_n]^2$, has a homogeneous
set in \scr X, so it isn't really a counterexample.  This contradiction 
completes the proof of the lemma.
\qed
\end{pf}

\medskip
Of course, we can iterate the operation $\rho$.  The lemma implies that, 
if \scr X is rich, then so is $\rho^n\scr X$ for any finite $n$.  Notice also 
that we have $\scr X\supseteq\rho\scr X\supseteq \rho^2\scr X\supseteq\dots$.

We shall apply all this information to a particular \scr X, namely
\[
\scr L=\{A\text{ finite, nonempty }\subseteq\N\,\mid\,\min(A)+2<|A|\},
\]
which is obviously rich.  Observe that any $A\in\scr L$ has $|A|\geq3$.  
It easily follows that any $A\in\rho^n\scr L$ has $|A|\geq3+n$.  
(In fact the sizes of sets in $\rho^n\scr L$ grow very rapidly, but we 
don't need this fact here.)  In particular, no finite set can belong to 
$\rho^n\scr L$ for arbitrarily large $n$, and so we can define a norm for 
finite sets by
\[
\nu:[\N]^{<\og}\to\N:A\mapsto\text{ least }n\text{ such that }
A\notin\rho^n\scr L.
\]

Because each $\rho^n\scr L$ is rich, we can partition $\N$ into consecutive 
finite intervals $I_n$ such that $I_n\in\rho^n\scr L$ for each $n$.  Define 
$p:\N\to\N$ to be the function sending all elements of any $I_n$ to $n$; so 
$p^{-1}[B]=\bigcup_{n\in B}I_n$ for all $B\subseteq\N$.

For any $X\subseteq\N$, define its growth  $\gamma(X):\N\to\N$ to be the 
sequence of norms of its intersections with the $I_n$'s:
\[
\gamma(X)(n)=\nu(X\cap I_n).
\]
Notice that, by our choice of the $I_n$, $\gamma(\N)(n)>n$ for all $n$.

\smallskip
With these preliminaries, we are ready to return to ultrafilters and prove the 
theorem. The proof uses ideas from \cite{bl74} and \cite{ros}, but some 
modifications are needed, and so we present the proof in detail.

\medskip

{\bf Proof of Theorem~\ref{constr}.}
Assume that \CH\ holds, and let $\U$ be an arbitrary selective ultrafilter
on 
$\N$.  We adopt the quantifier notation for ultrafilters: 
\begin{center}
    $(\scr Un)\,\phi(n)$ means ``for \scr U-almost all $n$, $\phi(n)$ holds,'' 
    \ie, $\{n\mid\phi(n\}\in\scr U$.
\end{center}
Call a subset $X$ of $\N$ \emph{large} if
\[
(\forall k\in\N)(\scr U n)\,\gamma(X)(n)>\sqrt n+k;
\]
equivalently, in the ultrapower of $\N$ by \scr U, $[\gamma(X)]$ is 
infinitely larger than $[\lceil\sqrt n\,\rceil]$.  Since $n$ is asymptotically 
much larger than $\sqrt n$, we have that $\N$ is large.

Using \CH, list all partitions $F:[\N]^2\to\{0,1\}$ in a sequence 
\sq{F_\alpha\mid\alpha<\aleph_1}
of length $\aleph_1$.  We intend to build a sequence 
\sq{A_\alpha\mid\alpha<\aleph_1} of subsets of $\N$ with the following properties:
\begin{lsnum}
\item Each $A_\alpha$ is a large subset of $\N$.
\item If $\alpha<\beta$, then $A_\beta\subseteq A_\alpha$ modulo \scr 
U, \ie, $p[A_\beta-A_\alpha]\notin\scr U$, \ie, $(\scr Un)\,A_\beta\cap 
I_n\subseteq A_\alpha$.
\item For each $n$, $F_\alpha$ is constant on $[A_{\alpha+1}\cap I_n]^2$.
\end{lsnum}
After constructing this sequence, we shall show how it yields the desired 
ultrafilter \scr V.

We construct $A_\alpha$ by induction on $\alpha$, starting with $A_0=\N$. 
We have already observed that requirement~(1) is satisfied by $\N$; the other 
two requirements are vacuous at this stage.

Before continuing the construction, notice that the relation of inclusion 
modulo \scr U is transitive.

At a successor step, we are given the large set $A_\alpha$ and we must find 
a large $A_{\alpha+1} \subseteq A_\alpha$ modulo \scr U such that $F_{\alpha}$ 
is constant on each $[A_{\alpha+1}\cap I_n]^2$.  Transitivity and the 
induction hypothesis then ensure that $A_{\alpha+1}$ is included in each 
earlier $A_\xi$ modulo \scr U.  
(We shall actually get $A_{\alpha+1}\subseteq A_\alpha$, not just modulo 
\scr U, but the inclusions in the earlier $A_\xi$'s will generally be only 
modulo \scr U.)  We define $A_{\alpha+1}$ by defining its intersection with 
each $I_n$; then of course $A_{\alpha+1}$ will be the union of all these 
intersections.

If $\gamma(A_\alpha)(n)\leq1$, then set $A_{\alpha+1}\cap I_n=\emp$.  Note that 
the set of all such $n$'s is not in \scr U, because $A_\alpha$ is large.  If 
$\gamma(A_\alpha)(n)>1$, then $A_\alpha\cap I_n$ is in 
$\rho^{\gamma(A_\alpha)(n)-1}\scr L$, so it has a subset that is homogeneous 
for $F_\alpha\restr[A_\alpha\cap I_n]^2$ and is in 
$\rho^{\gamma(A_\alpha)(n)-2}\scr L$; let $A_{\alpha+1}\cap I_n$ be such 
a subset.

This choice of $A_{\alpha+1}\cap I_n$ (for each $n$) clearly ensures that 
requirements~(2) and (3) are preserved.  For requirement~(1), simply observe 
that $(\scr Un)\,\gamma(A_{\alpha+1})(n)\geq\gamma(A_\alpha)(n)-1$.

For the limit step of the induction, suppose $\beta$ is a countable limit 
ordinal and we already have $A_\alpha$ for all $\alpha<\beta$.  Choose an 
increasing $\N$-sequence of ordinals \sq{\alpha_i:i\in\N} with limit
$\beta$.  
%%%AB: Changed $i<\N$ to $i\in\N$.  ($<$ was fine before $\omega$, but
%%%it looks strange before $\N$.
%%%OK
Let $A'_n=\bigcap_{i\leq n}A_{\alpha_i}$.  So the sets $A'_n$ form a 
decreasing sequence.  Because of induction hypothesis~(2), each $A'_n$ is 
equal modulo \scr U to $A_{\alpha_n}$ (i.e., each includes the other modulo
\scr U) and is therefore large.  We shall find a large set $A_{\beta}$ that 
is included modulo \scr U in all of the $A'_n$, hence in all the $A_{\alpha_n}$, 
and hence in all the $A_\alpha$ for $\alpha<\beta$.  Thus, we shall preserve 
induction hypotheses~(1) and (2); requirement~(3) is vacuous at limit stages.

We shall obtain the desired $A_\beta$ by defining its intersection with every 
$I_n$.

Because the ultrapower of $\N$ by \scr U is countably saturated, we can 
fix a function $g:\N\to\N$ such that its equivalence class $[g]$ in the 
ultrapower is below each $[\gamma(A'_n)]$ but above 
$[x\mapsto\lceil\sqrt x\,\rceil+k]$ for every $k\in\N$.  
For each $x\in\N$, define $h(x)$ to be the largest number $q\leq x$ such 
that $g(x)\leq\gamma(A'_q)(x)$, or 0 if there is no such $q$.  Finally, 
set $A_\beta\cap I_x =A'_{h(x)}\cap I_x$.  We verify that this choice of 
$A_\beta$ does what we wanted.

For any fixed $n$, \scr U-almost all $x$ satisfy $g(x)\leq\gamma(A'_n)(x)$, 
and $x\geq n$, and therefore $h(x)\geq n$, and therefore 
$A'_{h(x)}\cap I_x\subseteq A'_n\cap I_x$.  Thus, 
$A_\beta\subseteq A'_n$ modulo \scr U for each $n$.  
As we saw earlier, this implies requirement~(2).

Furthermore, for \scr U-almost all $x$,
\[
\gamma(A_\beta)(x)=\gamma(A'_{h(x)})(x)\geq g(x),
\]
%%% AB: deleted two occurrences of $\cap I_n$ from the preceding
%%% formula, and added a prime on the second $A$.  (It is an unsolved
%%% problem to figure out what I could have been thinking when I wrote
%%% the previous version.)
%%%OK
and so our choice of $g$ ensures that $A_\beta$ is large.

This completes the construction of the sequence \sq{A_\alpha:\alpha<\aleph_1} 
and the verification of properties~(1), (2), and (3).  We shall now use this 
sequence to construct the desired ultrafilter.

We claim first that the sets $A_\alpha\cap p^{-1}[B]$, where $\alpha$ ranges over $\aleph_1$ and $B$ ranges over \scr U, constitute a filter base.  Indeed, the intersection of any two of them, say
\[
A_\alpha\cap p^{-1}[B]\cap A_{\alpha'}\cap p^{-1}[B']
\]
with, say, $\alpha\leq\alpha'$, includes 
$A_{\alpha'}\cap p^{-1}[B\cap B'\cap C]$, where, thanks to requirement~(2) 
above, $C$ is a set in \scr U such that $A_{\alpha'}\cap I_n\subseteq A_\alpha$ 
for all $n\in C$.

Let \scr V be the filter generated by this filterbase.  Because each $A_\alpha$ 
is large, because $[\gamma(A_\alpha\cap p^{-1}[B])]=[\gamma(A_\alpha)]$ in 
the \scr U-ultrapower for any $B\in\scr U$, and because largeness is 
obviously preserved by supersets, we know that every set in \scr V is large.

We claim next that \scr V is an ultrafilter.  To see this, let $X\subseteq\N$ 
be arbitrary, and consider the following partition $F:[\N]^2\to \{0,1\}$.  
If $X$ contains both or neither of $x$ and $y$, then $F(\{x,y\})=0$; otherwise 
$F(\{x,y\})=1$.  Requirement~(3) of our construction provides an 
$A_\alpha\in\scr V$ such that $F$ is constant on each $[A\cap I_n]^2$, say 
with value $f(n)$.  A set on which $F$ is constant with value 1 obviously 
contains at most two points, one in $X$ and one outside $X$.  As $A$ is large, 
we infer that $(\scr Un)\,f(n)=0$.  That is, for $\scr U$-almost all $n$, 
$A\cap I_n$ is included in either $X$ or $\N-X$.  As \scr U is an ultrafilter, 
it contains a set $B$ such that either $A\cap I_n\subseteq X$ for all $n\in B$ 
or $A\cap I_n\subseteq \N-X$ for all $n\in B$. Then $A\cap p^{-1}[B]\in\scr V$ 
is either included in or disjoint from $X$.  Since $X$ was arbitrary, this 
completes the proof that \scr V is an ultrafilter.

Since all sets in \scr V are large, the finite-to-one function $p$ is not 
one-to-one on any set in \scr V. Thus \scr V is not selective, in fact not 
even a Q-point.

Our next goal is to prove that $\N\to[\scr V]^2_3$.  Let an arbitrary 
$F:[\N]^2\to\{0,1,2\}$ be given.  We follow the custom of calling the values 
of $F$ colors, and we use the notation $\{a<b\}$ to mean the set $\{a,b\}$ and 
to indicate the notational convention that $a<b$.  We shall find two sets 
$X,Y\in\scr V$ such that all pairs $\{a<b\}\in[X]^2$ with $p(a)=p(b)$ have a 
single color and all pairs $\{a<b\}\in[Y]^2$ with $p(a)\neq p(b)$ have a single 
color (possibly different from the previous color).  Then $X\cap Y\in\scr V$ 
has the weak homogeneity property required by $\N\to[\scr V]^2_3$.

To construct $X$, begin by considering the partition 
$G:[\N]^2\to\{0,1\}$ obtained 
from $F$ by identifying the color 2 with 1.  By our construction of \scr V, 
it contains a set $A_\alpha$ such that, for each $n$, all pairs in 
$[A_\alpha\cap I_n]^2$ are sent to the same color $g(n)$ by $G$.  
\scr U, being an ultrafilter, contains a set $B$ on which $g$ is constant.  
Then $A_\alpha\cap p^{-1}(B)$ is a set in \scr V such that all pairs $\{a<b\}$ in 
$A_\alpha\cap p^{-1}(B)$ with $p(a)=p(b)$ have the same $G$-color.  So 
the $F$-colors 
of these pairs are either all 0 or all in $\{1,2\}$.  If they are all 0, then 
$A_\alpha\cap p^{-1}(B)$ serves as the desired $X$.  If they are all in $\{1,2\}$, 
then we repeat the argument using $G'$, obtained from $F$ by 
identifying 2 with 0.  We obtain a set $A_{\alpha'}\cap p^{-1}(B')\in\scr V$ such 
that the $F$-colors of its pairs with $p(a)=p(b)$ are either all 1 or all 
in $\{0,2\}$.  Then $A_\alpha\cap p^{-1}(B)\cap A_{\alpha'}\cap 
p^{-1}(B')\in\scr V$ serves 
as the desired $X$.

It remains to construct $Y$.  It is well known that selective ultrafilters have 
the Ramsey property.  So we can find a set $B\in\scr U$ such that all or none 
of the pairs $\{x<y\}\in[B]^2$ satisfy the inequality
\[
9^{|\bigcup_{z\leq x} I_z|} <y.
\]
If we had the ``none'' alternative here, then all elements $y\in B$ would be 
bounded by $9^{|\bigcup_{z\leq x}I_z|}$ where $x$ is the smallest element of $B$.  That is absurd, as $B$ is infinite, so we must have the ``all'' alternative.

For each $b\in p^{-1}[B]$, let $f_b$ be the function telling how $b$ is related 
by $F$ to elements in earlier fibers over $B$.  That is, if $p(b)=n\in B$, let 
the domain of $f_b$ be $\{a\in\N:p(a)<n\text{ and }p(a)\in B\}$, and define 
$f_b$ on this domain by $f_b(a)=F(\{a,b\})$.  Notice that the domain of $f_b$
has cardinality at most $|\bigcup_{z\leq m}I_z|$ where $m$ is the last element 
of $B$ before $n$ (or 0 if $n$ is the first element of $B$).  Since $f_b$ takes 
values in 3, the number of possible $f_b$'s, for $p(b)=n\in B$, is at most
\[
3^{|\bigcup_{z\leq m}I_z|}<\sqrt n,
\]
by the homogeneity property of $B$.

Define a new partition, $H:[\N]^2\to\{0,1\}$ by setting $H(\{b<c\})=0$ if $f_b$ 
and $f_c$ are defined and equal, and $H(\{b<c\})=1$ otherwise.  Proceeding as 
in the first part of the construction of $X$ above, we obtain a set 
$Z\in\scr V$ such that all pairs $\{b<c\}\in[Z]^2$ with $p(b)=p(c)$ have 
the same $H$-color.  That is, in each of the sets $Z\cap I_n$ ($n\in B$), 
either all the points $b$ have the same $f_b$ or they all have different 
$f_b$'s.  But $Z$ is large so the number of such points is, for \scr U-almost 
all $n$, larger than $\sqrt n$.  So, by the estimate above of the number of 
$f_b$'s, there are not enough of these functions for every $b$ to have a 
different $f_b$.  Thus, for \scr U-almost all $n\in B$, all $b\in Z\cap I_n$ 
have the same $f_b$.  Shrinking $B$ to a smaller set in \scr U (which we still 
call $B$ to avoid extra notation), we can assume that, for all $n\in B$, $f_b$ 
depends only on $p(b)$ as long as $b\in Z$.  Let us also shrink $Z$ to 
$Z\cap p^{-1}[B]$, which is of course still in \scr V.

Going back to the original partition $F$, we have that the color of a pair 
$\{a<b\}\in[Z]^2$ with $p(a)<p(b)$ depends only on $a$ and $p(b)$, because 
this color is $F(\{a<b\})=f_b(a)$ and $f_b$ depends only on $p(b)$.

For each $a\in Z$, let $g_a$ be the function, with domain equal to the part 
of $B$ after $p(a)$, such that $g_a(n)$ is the common value of $F(\{a<b\})$ 
for all $b\in Z\cap I_n$.  This $g_a$ maps a set in \scr U (namely a final
segment of $B$) into 3, so it is constant, say with value $j(a)$, on some 
set $C_a\in\scr U$.

Using again the Ramsey property of \scr U, we obtain a set $D\in\scr U$ such 
that $D\subseteq B$ and all or none of the pairs $\{x<y\}\in[D]^2$ satisfy
\[
(\forall a\in B\cap I_x)\,y\in C_a.
\] 
If we had the ``none'' alternative, then, letting $x$ be the first element 
of $D$, we would have that
\[
D\cap\bigcap_{a\in B\cap I_x}C_a=\emptyset.
\]
But this is the intersection of finitely many sets from \scr U, so it
cannot be empty.  This contradiction shows that we must have the
``all'' alternative.  In view of the definition of $C_a$, this means
that, when $a<b$ are in $Z\cap p^{-1}[D]$ and $p(a)<p(b)$, 
%%%AB: Added the condition that $p(a)<p(b)$.  (It makes $a<b$
%%%redundant, but I think that does no harm; replacing $a<b$ with "$a$
%%%and $b$" wouldn't save any space or thought.
%%%OK

the $F$-color of $\{a<b\}$ depends only on $a$, not on $b$.  Since
there are only finitely many possible colors and since \scr V is an
ultrafilter, we can shrink $Z\cap p^{-1}[D]$ to a set $Y\in\scr V$ on
which the $F$-color of all such pairs is the same.

This completes the construction of the desired $Y$ and thus the proof that 
$\N\to[\scr V]^2_3$.

This partition relation, together with an easy induction argument, gives the 
following slightly stronger-looking result.   For any partition of 
$[\N]^2$ into any finite number of pieces, there is a set $H\in\scr V$ 
such that $[H]^2$ is included in the union of two of the pieces.  We shall 
need this for a partition into 6 pieces.

Finally, we prove that \scr V is quasi-selective.  Consider an arbitrary 
$f:\N\to\N$ such that $f(n)\leq n$ for all $n$; we seek a set in \scr V on 
which $f$ is non-decreasing.  Define a partition $F:[\N]^2\to6$ by setting 
\[
F(\{a<b\})=
\begin{cases}
0 &\text{if }p(a)=p(b)\text{ and }f(a)<f(b)\\
1 &\text{if }p(a)=p(b)\text{ and }f(a)=f(b)\\
2 &\text{if }p(a)=p(b)\text{ and }f(a)>f(b)\\
3 &\text{if }p(a)<p(b)\text{ and }f(a)<f(b)\\
4 &\text{if }p(a)<p(b)\text{ and }f(a)=f(b)\\
5 &\text{if }p(a)<p(b)\text{ and }f(a)>f(b)\\
\end{cases}
\]
Since $p$ is non-decreasing, the six cases cover all the possibilities.  Fix 
a set $H\in\scr V$ on whose pairs $F$ takes only two values.  The first of 
those two  values must be in $\{0,1,2\}$  and the second in $\{3,4,5\}$ 
because $p$ is neither one-to-one nor constant on any set in \scr V.  

If the second value were 5, then by choosing an infinite sequence 
$a_0<a_1<\dots$ in $H$ with $p(a_0)<p(a_1)<\dots$, we would get an infinite 
decreasing sequence $f(a_0)>f(a_1)>\dots$ of natural numbers.  Since this is 
absurd, the second value must be 3 or 4.

If the second value is 4, then any two elements of $H$ from different $I_n$'s 
have the same $f$ value.  But then the same is true also for any two elements 
of $H$ from the same $I_n$, because we can compare them with a third element 
of $H$ chosen from a different $I_m$.  So in this case $f$ is constant on $H$; 
in particular, it is non-decreasing, as desired.

So from now on, we may assume the second value is 3; $f$ is increasing on 
pairs in $H$ from different $I_n$'s.

Thus, if the first value is either 0 or 1, then $f$ is non-decreasing on all 
of $H$, as desired.  It remains only to handle the case that the first 
value is 2; we shall show that this case is impossible, thereby completing 
the proof of the theorem.

Suppose, toward a contradiction, that the first value were 2.  This means 
that, for each $n$, the restriction of $f$ to $H\cap I_n$ is strictly 
decreasing.  Temporarily fix $n$, and let $b$ be the smallest element of 
$H\cap I_n$.  Then $f(b)$ is the largest value taken by $f$ on $H\cap I_n$, 
and there are exactly $|H\cap I_n|$ such values.  Therefore, 
$f(b)\geq|H\cap I_n|-1$.  On the other hand, by the hypothesis on $f$, 
we have $f(b)\leq b$, and therefore $b\geq|H\cap I_n|-1$.

Now un-fix $n$.  We have just shown that 
\[
\min(H\cap I_n)\geq|H\cap I_n|-1,
\]
and so $H\cap I_n\notin\scr L$.  That is, 
\[
\gamma(H)(n)=\nu(H\cap I_n)=0
\]
for all $n$.  That contradicts the fact that, like all sets in \scr V, $H$ is 
large, and the proof of the theorem is complete.

\qed

\smallskip	   
Now, in order to deduce Theorem \ref{blch}, we have only to recall 
the well known fact that, under \CH, there are $2^{\ck}$ pairwise 
non-isomorphic selective ultrafilters. For each of them, say $\U$, Theorem 
\ref{constr} provides a non-selective quasi-selective $\V$ that is
Rudin-Keisler above $\U$. As remarked at the beginning of this section, 
according to \cite{bl74},
 there is a 
unique class of selective ultrafilters Rudin-Keisler below $\V$, 
which has been chosen so as to satisfy the partition relation
$\N\to[\scr V]^2_3$.
So all $\V$s are pairwise non-isomorphic.

%%%AB: That the $\V$'s are pairwise non-isomorphic is correct and was
%%%already stated at the bottom of page 7, but that was quite long ago
%%%(we're now on page 13), so it might be good to supply a reminder.
%%%The crucial point is that a weakly selective ultrafilter has only
%%%one isomorphism-class of selective ultrafilters below it.
%%%DONE
Finally, as remarked at the end of Section \ref{qsu}, 
Proposition \ref{blp} allows for associating to each $\V$ an isomorphic 
non-quasi-selective P-point $\V'$.

\section{Equinumerosity of point sets}\label{equi}

In this section we study a notion of ``numerosity''
for \emph{point sets of natural numbers},
\emph{i.e.} for subsets of the spaces of $k$-tuples $\N^k$.
This numerosity will be defined by starting from
an equivalence relation of ``equinumerosity'' that
satisfies all the basic properties of equipotency
between finite sets.

For simplicity we follow the usual practice
and we identify Cartesian products with the corresponding
``concatenations''. That is, for every $A\subseteq\N^k$
and for every $B\subseteq\N^h$, we identify
$A\times B=\{(\vec{a},\vec{b})\mid \vec{a}\in A, \vec{b}\in B\}$
with:
$$A\times B=
\{(a_1,\ldots,a_k,b_1,\ldots,b_h)\mid
(a_1,\ldots,a_k)\in A\ \text{and}\ (b_1,\ldots,b_h)\in B\}.$$

\begin{definition}\label{eqn}
We call \emph{equinumerosity} an equivalence relation $\approx$
that satisfies the following properties
for all point sets $A,B$ of natural numbers:

\begin{enumerate}
\item[\euno]
$A\approx B$ if and only if $A\setminus B\approx B\setminus A$.
\item[\edue]
Exactly one of the following three conditions holds:
\begin{enumerate}
\item[(a)]
$A\approx B$\,;
\item[(b)]
$A'\approx B$ for some proper subset $A'\subset A$\,;
\item[(c)]
$A\approx B'$ for some proper subset $B'\subset B$.
\end{enumerate}
\item[\etre]
$A\* \{P\}\approx \{P\}\times A \eq A$ for every point $P$.
\item[\ecinque]
$A\approx A'\ \&\ B\approx B'\ \Rightarrow\
A\times B\approx A'\times B'$.

\end{enumerate}
\end{definition}

Some comments are in order.
Axiom \euno\ is but a compact
equivalent reformulation of the second and third common notions 
of Euclid's Elements (see \cite{eu}):
\begin{quotation}
    \emph{``If equals be added to equals,
    the wholes are equal''},
\end{quotation}
 and
\begin{quotation}
    \emph{``if equals be subtracted from equals, the
    remainders are equal''}.
\end{quotation}
 (A precise statement of this equivalence is 
given in Proposition \ref{sumdif} below.)
Notice that the first common notion 
\begin{quotation}
    \emph{``Things which are equal to the same 
    thing are also equal to one another''}
\end{quotation}
 is already secured by the assumption 
that equinumerosity is an equivalence relation.

The \emph{trichotomy property} of axiom \edue\ combines two natural
ideas: firstly that, given two sets, one is equinumerous to some 
subset of the other, and secondly 
that  no proper subset is equinumerous to the  
set itself. So \edue\ allows for a natural ordering of sizes that 
satisfies  the implicit assumption of the classical 
theory that \emph{(homogeneous) magnitudes are always comparable}, as well as  the fifth 
Euclidean common notion 
\begin{quotation}
    \emph{``The whole is greater than the 
    part''}.
\end{quotation}
We remark that both properties \euno\ and \edue\ hold
 for equipotency between finite sets, 
while both  fail badly for equipotency between infinite sets.

The third axiom \etre\ 
 can be viewed 
 as an instance of the 
 fourth Euclidean common notion 
 \begin{quote}
     \emph{``Things applying $[$exactly$]$ onto one 
      another are equal to one another''}.\footnote{~Equicardinality is 
      characterized by equipotency; similarly one might desire that
      equinumerosity  be 
      characterized by ``isometry'' witnessed by a suitable class of 
      bijections. This assumption seems \emph{prima facie} very demanding. 
      However, we shall see that it is fulfilled by the \emph{asymptotic 
      equinumerosities} of Section \ref{asy}.}
 \end{quote}
 In particular \etre\ incorporates the idea that any set has equinumerous 
 ``lifted copies'' in any higher 
 dimension (see Proposition \ref{lifdim} below).
 
  Axiom \ecinque\ is postulated so as  to allow for the natural 
 definition of \emph{multiplication} of numerosities, which
  admits  
  the numerosity of every singleton as an 
 identity by axiom \etre.
This multiplication, together with the natural addition of 
numerosities, as 
given by \emph{disjoint union}, 
  satisfy the properties of \emph{discretely ordered semirings}
 (see Theorem \ref{ring}). 
 
 Remark that we do not postulate here 
 \emph{commutativity} of product. On the one hand, this assumption is 
 unnecessary for the general treatment of numerosities; on the other 
 hand, commutativity follows from the given axioms in the case of 
 \emph{asymptotic} numerosities (see Section \ref{asy}).

\begin{proposition}\label{sumdif}
The Axiom \emph{\euno}\ is equivalent to the
conjunction of the following two principles:\footnote
{~We implicitly assume that $A,B\subseteq\N^k$
and $A',B'\subseteq\N^{k'}$ are \emph{homogeneous} pairs
(\ie~sets of the same dimension);
%%%AB: changed comma to semicolon
%%%OK
otherwise $A\cup B, A'\cup B', A\setminus B, A'\setminus B'$
are not point sets.}
\begin{enumerate}
\item
\emph{Sum Principle:}\
Let $A,A',B,B'$ be such that $A\cap B=\emptyset$
and $A'\cap B'=\emptyset$.
\\
If $A\approx A'$ and $B\approx B'$, then
$A\cup B\approx A'\cup B'$.
\item
\emph{Difference Principle:}\
Let $A,A',C,C'$ be such that
$A\subseteq C$ and $A'\subseteq C'$.
\\
If $A\approx A'$ and $C\approx C'$, then
$C\setminus A\approx C'\setminus A'$.
\end{enumerate}
\end{proposition}

\begin{proof}
    We begin by proving that $\euno$ follows from the conjunction of 
    $(1)$ and $(2)$. In fact, if $A\/ B \eq B\/ A$, then $$A=(A\/ B) 
    \cup (A\cap B) \eq  (B\/ A) 
    \cup (A\cap B) =B,$$ by $(1)$. Conversely, if $A\eq B$, then 
    $$A\/ B= A\/ (A\cap B) \eq B\/ (A\cap B) =B\/ A,$$ by $(2)$.

    Now remark that, if we put $A\cup B=C$ and $A'\cup B'=C'$ in (1), 
   then both principles follow at once from the statement
    \begin{itemize}
        \item[$\euno^{*}$]\emph{ Let $A,A',C,C'$ be such that 
$A\subseteq C$ and $A'\subseteq C'$.  If $A\approx A'$, 
then}
$$C\setminus A\approx C'\setminus A'\ \ \Iff\ \ C\approx C'.$$
    \end{itemize}
    
    We are left to show that $\euno$ implies $\euno^{*}$.

\smallskip
Assume that the equivalence relation
$\approx$ satisfies $\euno$, and put

\smallskip\noindent
$A_{0}=A\/ C'$,\ \  $A'_{0}=A'\/ C$,\ \  $D=A\cap A'$,\ \ 
$A_{1}=A\/ (A_{0}\cup D)$,\ \  $A'_{1}=A'\/ (A'_{0}\cup D)$,

\smallskip\noindent
$C_{0}=C\/ (A\cup C')$, $C'_{0}=C'\/ (A'\cup C)$,
$E=(C\cap C')\/(A\cup A')$,

\smallskip\noindent
so as to obtain pairwise disjoint sets $A_{0},A_{1},A'_{0},A'_{1},D, 
C_{0},C'_{0},E$ such that

\smallskip\noindent
$A=A_{0}\cup A_{1}\cup D$,\ \
$A'=A'_{0}\cup A'_{1}\cup D$,\ \

\smallskip\noindent
$C\/ C'=A_{0}\cup C_{0}$,\ \ 
$C'\/ C=A'_{0}\cup C'_{0}$,\ \ 

\smallskip\noindent
$C\/ A=C_{0}\cup E\cup A'_{1}$,\ \ 
$C'\/ A'=C'_{0}\cup E\cup A_{1}$.

\smallskip
Hence
$$C\eq C'\ \Iff\ \ C_{0}\cup A_{0}\eq C'_{0}\cup A'_{0} $$
and 
$$C\/ A  \eq C'\/ A' \ \ \Iff\ \ C_{0}\cup A'_{1}\eq C'_{0}\cup 
A_{1}$$

Since $A\eq A'$ we have\ \  $A_{0}\cup A_{1}\eq A'_{0}\cup A'_{1}$,\ 
whence $$C_{0}\cup A_{0}\cup A_{1} \eq C_{0}\cup A'_{0}\cup A'_{1}.$$

\smallskip\noindent
So $C\/ A  \eq C'\/ A'$ implies $C_{0}\cup A_{0}\cup A_{1}\eq 
C'_{0}\cup A_{1}\cup A'_{0}
$, whence $C\eq C'$.

\smallskip\noindent
Conversely, $C\eq C'$ implies $C_{0}\cup A_{0}\cup A_{1}
\eq C'_{0}\cup A'_{0}\cup A_{1}$, whence $C\/ A  \eq C'\/ A'$.

\qed

\end{proof}

\medskip
Thus equinumerosity behaves coherently with respect to
the operations of \emph{disjoint union} and \emph{set difference}.
We can now prove that our notion of equinumerosity satisfies the 
basic requirement
that \emph{finite point sets are equinumerous if and only if they 
have the same 
``number of elements''}. 
\begin{proposition}\label{fin}
	Assume \emph{\euno}\ and \emph{\edue}, and let $A,B$ be finite sets. Then
	$$A\approx B\ \Iff\ |A|=|B|.$$
\end{proposition}

{\bf Proof.}
~We begin by proving that \edue\ implies that any two singletons are equinumerous.
In fact the only proper subset of any singleton is the empty set, and
 no nonempty set can be equinumerous to $\0$, by \edue.

Given nonempty finite sets $A,B$,
pick $a\in A$, $b\in B$ and
put $A'=A\/\{a\}$, $B'=B\/\{b\}$.  Then, by \euno$^{*}$,
 $A'\approx B'\Iff A\approx B$, because $\{a\}\approx\{b\}$, and
the thesis follows by  induction on $n=|A|$.

\qed

\medskip
By the above proposition, we can identify
each natural number $n\in\N$ with the
equivalence class of all those point sets
that have finite cardinality $n$.

Remark that trichotomy is not essential in order to 
obtain Proposition \ref{fin}. In fact, let $\eq$ be nontrivial 
and satisfy \euno,  \etre, 
and \ecinque. Then, by \etre, we have 
$$\{x\}\eq\{y\}\*\{x\}\eq\{x\}\*\{y\}
\eq \{y\}$$
for all $x\in \N^{h}$ and all $y\in\N^{k}$, and so all singletons are 
equinumerous. 
On the other hand, if any singleton is equinumerous to $\0$, then all 
sets are, because of \ecinque.

\smallskip
Clearly \etre\ formalizes the natural
idea that singletons have ``unitary'' numerosity.
A trivial but important consequence of this 
axiom, already mentioned above, is the existence of infinitely many 
pairwise disjoint equinumerous copies of any given point set in every 
higher dimension. Namely

\begin{proposition}\label{lifdim}
 Assume \emph{\etre}, and let $A\subseteq\N^k$ be a $k$-dimensional point set. 
 If $h>k$ let 
 $P\in\N^{h-k}$ be any $(h-k)$-dimensional point. Then
$\{P\}\* A\incl \N^{h}$ is equi\-numerous to $A$.
\qed
\end{proposition}

\section{The algebra of numerosities}\label{arit}

\smallskip
Starting from the equivalence relation of \emph{equipotency},
Cantor introduced the algebra of cardinals by means
of disjoint unions and Cartesian products. Our axioms have been 
chosen so as to allow for the introduction of
 an ``algebra of numerosities'' in the same 
Cantorian manner.

\begin{definition}\label{num}
\emph{Let $\WW=\bigcup_{k\in\N^{+}} \P(\N^{k})$ be the family of all 
point sets over $\N$, and let $\eq$ be an equinumerosity relation on $\WW$.}

\noindent
    $\bullet$ \emph{The \emph{numerosity} of $A\in\WW$ (with respect to $\eq$)
    is the equivalence class of all point sets
equinumerous to $A$}
$$ \nk_{\eq}(A)= [A]_{\eq}=\{B\in\WW \mid B\approx A\}.$$

\noindent
    $\bullet$   \emph{The  \emph{set of numerosities} of 
$\ \eq\ $ is the quotient set $\ \Ng_{\eq}=\WW/\eq$.
}

\noindent
    $\bullet$  \emph{The \emph{numerosity function} 
associated to $\ \approx\ $ is
the canonical map $\ \ng_{\eq}:\WW\to\Nk_{\eq}$.}

\end{definition}

{In the sequel we shall drop the 
subscript $\eq$ whenever the equinumerosity relation is fixed.}
Numerosities will be usually denoted by \emph{Frakturen}
$\xk,\yk,\zk,$ \emph{etc}.

The given axioms guarantee that numerosities are naturally equipped
with a ``nice'' algebraic structure.
(This has to be contrasted with the awkward cardinal algebra,
where \pes\ \
$\kappa+\mu=\kappa\cdot\mu=\max\{\kappa,\mu\}$ for all infinite
$\kappa,\mu$.)

\begin{Theorem}\label{ring}
    Let $\Nk$ be the set of numerosities of the equinumerosity 
    relation $\eq$. Then
there exist unique operations $\, +\, $ and $\, \cdot\, $,
and a unique linear order $\, <\, $ on $\,\Ng$, such that for
all point sets $A,B$:
\begin{enumerate}
\item
$\ng(A)+\ng(B)=\ng(A\cup B)$ whenever $A\cap B=\emptyset$\,;\footnote
{~Again, here we implicitly assume that $A,B\subseteq\N^k$
are homogeneous, as otherwise their union $A\cup B$ would not be a point set.}
\item
$\ng(A)\cdot\ng(B)=\ng(A\times B)$\,;
\item
$\ng(A)<\ng(B)$ if and only if $A\approx B'$
for some proper subset $B'\subset B$.
\end{enumerate}

The resulting structure on $\Ng$
is the non-negative part of a discretely ordered ring
$(\mathfrak{R},\, 0,\, 1,\, +,\, \cdot,\, <)$.
Moreover, if the fundamental subring of $\Rg$ is identified with 
$\Z$, then
$\ng(A)=|A|$ for every finite point set $A$.

\end{Theorem}

\begin{proof}
We begin with the ordering.
Trivially $A\approx A$, and so the trichotomy property \edue\ implies
the irreflexivity $\nk(A)\not<\nk(A)$. In order to prove transitivity,
we show first the following property:

\begin{center}
$(\star)$\quad
If  $A\approx B$, then for any $X\subset A$ there
exists $Y\subset B$ such that $X\approx Y$.
\end{center}

Since $A\approx B$, the proper subset $X\subset A$ cannot be
equinumerous to $B$. Similarly, $B$ cannot be
equinumerous to a proper subset $X'\subset X\subset A$.
We conclude that $X\approx Y$ for some $Y\subset B$,
and $(\star)$ is proved.

Now assume $\nk(A)<\nk(B)<\nk(C)$. Pick proper subsets $A'\subset B$ and
$B'\subset C$ such that $A\approx A'$ and $B\approx B'$.
By $(\star)$, there exists $A''\subset B'$ such that
$A'\approx A''$. So, $\nk(A)<\nk(C)$ holds because $A\approx A''\subset C$,
and transitivity follows. Finally, again by trichotomy,
we get that the order $<$ is linear.

\smallskip
We now define a sum for numerosities.
Given $\xk,\yk\in\Ng$, there exist
homogeneous disjoint point sets $A, B\subseteq\N^k$
such that $\ng(A)=\xk$
and $\ng(B)=\yk$, by Proposition \ref{lifdim}. Then put $\xk+\yk=\ng(A\cup B)$.
This addition is independent
of the choice of $A$ and $B$, by Proposition \ref{sumdif}. 
Commutativity and associativity
 trivially follow from the
corresponding properties of disjoint unions.

Clearly $0=\ng(\emptyset)$ is the (unique) neutral element. Moreover
Proposition \ref{sumdif} directly yields the 
cancellation law
$$\ \xk+\zk = \xk+\zk'\ \Iff\ \zk=\zk'.$$

By definition,
$\xk\le \yk$ holds if and only if
there exists $\zk$ such that $\yk=\xk +\zk$. Such a $\zk$ is
unique by the cancellation law, and so the monotonicity property 
with respect to addition 
   follows from  the equivalences
$$\wk+\xk \le \wk+\yk\ 
\Iff\ \exists \zk\, (\wk+\xk+\zk=\wk+\yk)\
\Iff\ \exists \zk\, (\xk+\zk=\yk)\ \Iff\ \xk\le
\yk.$$
The multiplication of numerosities given by condition (2) is well-defined 
by axiom \ecinque. 
Associativity follows from the corresponding
property of concatenation (and this is the reason for our
convention on Cartesian products).
   Distributivity
is inherited from the corresponding property of Cartesian products
with respect to disjoint unions.
 Moreover $1=\ng(\{P\})$ is an identity, by
\etre.
Finally, by definition, $\xk\cdot \yk = 0$ if and only if
$\xk=0$ or $\yk=0$.

\smallskip
Therefore $\langle\Ng,+,\cdot\,,<\rangle$
 is \emph{the non-negative part of a linearly ordered
ring} $\Rg$, say.\footnote{~$\Rg$ can be defined as usual by mimicking
the construction of $\mathbb{Z}$ from $\N$.
Take the quotient of $\Ng\times\Ng$ modulo the equivalence
$(x,y)\equiv(x',y')\Leftrightarrow x+y'=x'+y$, and define the
operations in the obvious way.}
By Proposition \ref{fin}, $\N$ can be identified with the set of the 
numerosities of finite 
sets; hence it is an initial segment of $\Ng$. 
It follows that the
ordering of $\Rg$ is discrete.
\end{proof}
\qed

\medskip

Recall that we have not postulated an axiom of commutativity, and so
 in general the set of numerosities might be a
\emph{noncommutative} semiring. However 
numerosities can be given a nice algebraic characterization as 
\emph{the non-negative part of the 
quotient of a ring of noncommutative formal power series with 
integer coefficients modulo a suitable prime ideal}. 
Let us fix our notation as follows:

\begin{itemize}

\item let $\TT=\la t_{n}\mid\, n\in\N\ra$ be a sequence of 
\emph{noncommutative} indeterminates and let $\Z\la\!\la \TT\ra\!\ra$ be 
the corresponding ring of \emph{noncommutative formal power series}; 

\smallskip
\item   to each  point
$x=(x_{1},\ldots,x_{k})\in\N^{k}$ associate
the \emph{noncommutative monomial}
$\ t_{x}=t_{x_{1}}\ldots t_{x_{k}};$

\item let $S_{X}$ be the \emph{characteristic series}
of the point set $X\in \WW$, \ie 
$$S_{X}=\sum_{x\in X}t_{x}\in\SZ;$$

\item let $\RR\incl\SZ$ be  the ring of
 all formal series of \emph{bounded degree} in $\TT$
with \emph{bounded integral coefficients}; 
\item let $\RP$ be the
multiplicative subset of the \emph{bounded positive  series}, \ie\ 
the series  in $\RR$ having only positive 
coefficients;

\item let $\Ik_{0}$ be the two-sided ideal of $\RR$
generated by $\{t_{n}-1\,\mid\, n\in\N\}$.

\end{itemize}

\bigskip
Then

\begin{proposition}\label{chars} ${}$
    
    \begin{enumerate}
        \item Characteristic series of nonempty point sets belong to 
    $\RP$; 
    
        \item characteristic series behave well with respect to
     \emph{unions} and  \emph{products}, \ie :
    $$\ \ S_{X}+S_{Y}=S_{X\cup Y}+S_{X\cap Y}\ \ \mbox{and}\ \
    \  S_{X\times Y}=S_{X}\cdot S_{Y}\ \mbox{for all}\ 
     X,Y\in\WW.\footnote{~\emph{Caveat}: if $X,Y$ have different 
     dimensions their union does not belong to $\WW$, so the first 
     equality does not apply.
     }$$    
     
\item Let $S=a+\sum a_{x}t_{x}\in\RR$, with $|a_{x}|\le 
B$ and $\deg t_{x}\le d$ for all $x$; then 
$$S=a+\sum_{k=1}^{d}\sum_{i=1}^{B} (S_{X_{ik}}-S_{Y_{ik}}),$$ 
where\ \ $ 
X_{ik}=\{x\in\N^{k} \mid a_{x} \ge i\},\ 
Y_{ik}=\{x\in\N^{k} \mid a_{x} \le -i\}.$

\smallskip
\item For any $S\in\Ik_{0}$ there exist  finite sequences $\la 
a_{h}\mid h\le n\ra$ of integers and $\la X_{i}\mid 1\le i\le m\ra$ of
point sets,
such that
\begin{eqnarray*}
     & S & =\ a_{0}(t_{0}-1) + \sum_{i,j=1}^{m} \eg_{0ij}
(S_{X_{i}\* \{0\}\* X_{j}}-S_{X_{i}\* X_{j}}) +
    \label{}  \\
     &  & +\ \sum_{i=1}^{m}\dg_{0i}(S_{X_{i}\* \{0\}}-S_{X_{i}}) + 
     \eta_{0i}(S_{ \{0\}\* X_{i}}-S_{X_{i}})+
    \label{}  \\
     &  & +\ \sum_{h=1}^{n} \left(a_{h}(t_{h}-t_{0}) +
\sum_{i,j=1}^{m}\eg_{hij}(S_{X_{i}\* \{h\}\* X_{j}}-S_{X_{i}\* 
\{0\}\* X_{j}}) +\right.
    \label{}  \\
     &  & +\ \left.\sum_{i=1}^{m}\dg_{hi}(S_{X_{i}\* \{h\}}-S_{X_{i}\* 
\{0\}})+ \eta_{hi}(S_{ \{h\}\* X_{i}}-S_{ 
\{0\}\* X_{i}})\right)
    \label{}
\end{eqnarray*}

with suitable coefficients\ $\eg_{hij},\dg_{hi},\eta_{hi}\in\{0,\pm1\}$.

\medskip
        \item Every positive series $P\in\RP$ is equivalent modulo $\Ik_{0}$ 
to the characteristic series $S_{X}$ of some nonempty point set $X$.

\smallskip

\item $\RR$ is the subring with 
identity of $\,\SZ$ generated by $\Ik_{0}$ together with
   the set of the characteristic series of all point sets. More 
   precisely 
   $$\RR/\Ik_{0}=\{S_{X}-S_{Y}+\Ik_{0}\,\mid X,Y\in\WW\}$$
    \end{enumerate}  
   
    \end{proposition}
\proof
$(1), (2)$ and $(3)$ directly follow from the definitions.

In order to prove $(4)$, choose the binomials $t_{n}-t_{0}$ together 
with 
$t_{0}-1$ as generators of $\Ik_{0}$. Recalling $(3)$, we can write each element 
of  $\Ik_{0}$ as the sum of finitely many terms of one of the following types:
\begin{itemize}
    \item[$(a)$]  $a(t_{0}-1)$ or $a(t_{n}-t_{0}),\, n>0$, with $a\in\Z$;

    \item[$(b)$]  $\pm S_{X}(t_{0}-1)$ or $\pm 
    S_{X}(t_{n}-t_{0}),
    \, n>0$ ;

    \item[$(c)$]  $\pm (t_{0}-1)S_{X}$ or $\pm 
    (t_{n}-t_{0})S_{X},
    \, n>0$;

    \item[$(d)$]  $\pm S_{X}(t_{0}-1)S_{Y}$ or $\pm 
    S_{X}(t_{n}-t_{0})S_{Y},
    \, n>0$.
\end{itemize}
In order to get $(4)$,  list all the point sets appearing above, each 
with the appropriate number of repetitions, and remark that
$S_{X}(t_{0}-1)S_{Y}=S_{X\* \{0\}\* Y}-S_{X\* Y}$, and similarly for 
the other types of summands.

$(6)$ follows 
immediately from $(5)$. As for $(5)$, let a positive series $P=\sum 
n_{x}t_{x}$ be given, with coefficients not exceeding $B$, and degrees 
not exceeding $d$.   
Factorize
each monomial as $t_{x}=t_{x'}t_{x''}$, where $t_{x'}\in\Sud$ is the largest 
initial part of $t_{x}$ containing only
the variables $t_{0},t_{1}$. Then 
$$P=\sum_{z}\left(\sum_{\substack{x\\ t_{x''}=t_{z}}}n_{x}t_{x'}\right)t_{z}.$$
Put $$N_{z}= \sum_{\substack{x\\ t_{x''}=t_{z}}}n_{x}.$$
Modulo $\Ik_{0}$ we may replace each internal sum 
$\sum n_{x}t_{x'}$
by an arbitrary sum 
of $N_{z}$ different monomials $t_{y}$ in the variables $t_{0},t_{1}$.
The resulting series is the characteristic series of a 
$k$-dimensional point set, 
provided that all the 
resulting monomials $t_{y}t_{z}$ have the same 
degree $k$.
Such a choice of the monomials $t_{y}$ is clearly possible by taking 
$k$ such that $2^{k}>B\cdot2^{d}$.

\qed

\medskip
 Once an equinumerosity relation has been fixed, the expression $(4)$ shows 
 that every element of $\Ik_{0}$ is the sum of differences of 
 characteristic series of equinumerous 
 point sets, plus a multiple of $t_{0}-1$.
 So numerosities can be viewed as elements of the quotient ring of 
$\RR$ modulo suitable ideals extending $\Ik_{0}$, namely ideals $\Ik$ such 
that the quotient $\RR/\Ik$  is a 
\emph{discretely ordered  ring} whose 
\emph{positive elements} are the cosets $P+\Ik$ for $P\in\RP$. 
To this aim we define:

\begin{definition}\label{gid}
    A two-sided ideal $\Ik$ of $\RR$ is a
   {\emph{gauge ideal}} if 
   \begin{itemize}
       \item  $\Ik_{0}\incl \Ik$,

       \item  $\RP\cap\Ik=\0$, and

       \item for all $S\in\RR\/\Ik$ there exists  
   $P\in\RP$ such that $\ S\pm P\in\Ik$. 
\end{itemize}  
\end{definition}

\noindent
Remark that any gauge ideal   is a \emph{prime} ideal of $\RR$ that is 
\emph{maximal} among the two-sided ideals disjoint from $\RP$.

We can now give a purely algebraic characterization of 
equinumerosities.

   \begin{thm}\label{ser}
       There exists a biunique correspondence between 
       equinumerosity relations on the
       space $\WW$ of all point sets over $\N$ and gauge ideals of 
       the ring $\RR$ of all formal series of bounded degree in $\SZ$ with bounded  coefficients. 
       In 
       this correspondence, if the equinumerosity $\eq$ corresponds to the 
       ideal  $\Ik$,  then
        \begin{equation*}
         X\eq Y\ \Iff\ \ S_{X}-S_{Y}\in\Ik. \tag{$*$}
            \label{(*)}
        \end{equation*}

       More precisely, let $\nk:\WW\to\Nk$ be the numerosity function associated 
       to $\eq$,
       and let $\pi:\RR\to\RR/\Ik$ be the canonical projection. Then
   there exists a unique isomorphism $\tau$ of ordered rings
   such that the following diagram commutes

   \medskip	   
   \bigskip
   
	   \begin{center}
	   \begin{picture}(180,60)
	      \put(0,0){\makebox(0,0){$\Nk$}}
	      \put(90,0){\makebox(0,0){$\Rk$}}
	      \put(90,33){\makebox(0,0){$(**)$}}
	      \put(180,0){\makebox(0,0){$\RR/\Ik$}}
	      \put(45,60){\makebox(0,0){$\WW$}}
	      \put(132,60){\makebox(0,0){$\RR$}}
	      \put(45,6){\makebox(0,0){$j$}}
	      \put(135,-6){\makebox(0,0){$\cong$}}
	      \put(135,6){\makebox(0,0){$\tau$}}
	      \put(12,36){\makebox(0,0){$\nk$}}
	      %\put(81,36){\makebox(0,0){$\nk_\W\!\!
	      %\upharpoonleft\!\!_{\mathbb{W}_0}$}}
	      \put(161,36){\makebox(0,0){$\pi$}}
	      \put(90,66){\makebox(0,0){$\Sg$}}
	      \put(16,0){\vector(1,0){54}}
	      \put(106,0){\vector(1,0){60}}
	      \put(40,52){\vector(-1,-1){40}}
	      %\put(46,52){\vector(1,-1){40}}
	      \put(135,52){\vector(1,-1){40}}
	      \put(57,60){\vector(1,0){62}}
	   \end{picture}
	   \end{center}

	   \bigskip

\noindent
$($where $\Sg$ maps any $X\in\WW$ to its characteristic
 series $S_{X}\in\RR,$ $\Rk$ is the ring generated by $\Nk$,
 and $j$ is the natural 
 embedding.$)$   
  
	     \end{thm}
	     
\proof
Given a gauge ideal $\Ik$ on $\RR$, define the equivalence $\eq$ on 
$\WW$ by 
$$X\eq Y\ \Iff\ \ S_{X}-S_{Y}\in\Ik.$$
We show that $\eq$ satisfies conditions \euno-\ecinque\ of an equinumerosity 
relation.\\
 \euno\ is trivial and \ecinque\ 
follows  from Proposition \ref{chars}$(2)$, because
$$S_{X\* Y}-S_{X'\* Y'}=(S_{X}-S_{X'})S_{Y}+S_{X'}(S_{Y}-S_{Y'}).$$  
\etre\ holds 
because $\Ik_{0}\incl\Ik$. We are left with
the trichotomy condition \edue. First of all observe that at most 
one of the conditions 
$ X\eq Y$, $X\eq X'\pincl Y$, and $Y\eq Y'\pincl X$ can hold. \Pes, by
assuming the first two conditions one would get  $$S_{Y\/X'}=S_{Y}-S_{X'}=
S_{Y}-S_{X} +S_{X}-S_{X'}\in\Ik\cap \RP,$$
against the second property of gauge ideals.
Finally, if $ 
S_{X}-S_{Y}\not\in\Ik$, then  by combining the third property of 
gauge ideals with Proposition \ref{chars}$(5)$, one obtains 
$S_{X}-S_{Y}\mp S_{Z}\in \Ik$ for some nonempty point set $Z$. Hence 
either $X\eq Y\/Z$ or $Y\eq X\/Z$.
  
  \smallskip

 Conversely, given an
 equinumerosity relation $\eq$,   let $\Ik$ be the two-sided ideal of $\RR$ 
 generated by 
 the set $\{S_{X}-S_{Y}\mid X\eq Y\}\cup \{t_{0}-1\}$. 
 Then  $\Ik_{0}\incl\Ik$, by Proposition \ref{chars}$(4)$. Moreover,
 by Proposition \ref{chars}.$(6)$, 
 every $S\in\RR$ is congruent modulo $\Ik$ to a difference 
 $S_{X}-S_{Y}$, which in turn is congruent to $\pm S_{Z}$ for some 
 $Z$, by property \edue. So $\Ik$ is gauge, provided that it is 
 disjoint from $\RP$.
   
  In order to prove that $\Ik\cap\RP=\0$, remark that,  
 by points $(2),(3),$ and 
 $(4)$ of Proposition \ref{chars},
 every element  $S\in\Ik$ 
 can be written as 
 $$S=a(t_{0}-1)+\sum_{i=1}^{m}(S_{X_{i}}-S_{Y_{i}})$$
 for suitable (not necessarily different) equinumerous point sets 
 $X_{i}\eq Y_{i}$.
 
 If no term in $S$ is negative, then $a\le 0$, and every 
 element of each $Y_{i}$ belongs to some $X_{j}$.  
 Partition $Y_{1}$ as 
 $Y_{1}=\bigcup_{j=1}^{m}Y_{1j}$, with $Y_{1j}\incl X_{j}$. 
 Put $X'_{i}=X_{i}\/ Y_{1i}$: then 
 $$S=a(t_{0}-1)+\sum_{i=1}^{m}S_{X'_{i}}-\sum_{i=2}^{m}S_{Y_{i}},\ \ 
 \mbox{and}\ \ \sum_{i=1}^{m}\nk(X'_{i})=\sum_{i=2}^{m}\nk(Y_{i}).$$
 So $Y_{2}\incl\bigcup_{i=1}^{m}X'_{i}$, and we can proceed as above by 
 subtracting each element of $Y_{2}$ from an appropriate $X'_{i}$, 
 thus obtaining sets $X''_{i}$ that satisfy the conditions
 $$S=a(t_{0}-1)+\sum_{i=1}^{m}S_{X''_{i}}-\sum_{i=3}^{m}S_{Y_{i}},\ \ 
 \mbox{and}\ \ \sum_{i=1}^{m}\nk(X''_{i})=\sum_{i=3}^{m}\nk(Y_{i}).$$
 Continuing this procedure with $Y_{3},\ldots,Y_{m}$ we reach the 
 situation
 $$S=a(t_{0}-1)+\sum_{i=1}^{m}S_{X^{(m)}_{i}},\ \ 
 \mbox{and}\ \ \sum_{i=1}^{m}\nk(X^{(m)}_{i})=0.$$
 Hence all sets $X^{(m)}_{i}$ are empty, and $S=a(t_{0}-1)$. But then 
 $a=0$, because $S$ has no negative terms. 
 So $\Ik\cap\RR^{+}=\0$, and the ideal $\Ik$ is gauge.
 
 In particular one obtains the condition $(*)$ 
 $$S_{X}-S_{Y}\in\Ik\ \Iff\ \ X\eq Y.$$
 In fact, if $X\not\eq Y$, say $X\eq X'\pincl Y$, then 
 $S_{X}-S_{X'}= S_{X}-S_{Y}+S_{Y\/X'}\in\Ik$. Hence $S_{X}-S_{Y}$ cannot 
 be in $\Ik$, because otherwise $S_{Y\/X'}\in\Ik\cap \RP$.
 So the 
 equinumerosity $\eq$ corresponding to the ideal $\Ik$ is precisely the one 
 we started with, and the correspondence between equinumerosities and 
 gauge ideals is biunique. 
     
 Given $\Ik$ and $\eq$, let $\nk:\WW\to\Nk$ be the numerosity function associated to $\eq$. 
 In order to make the diagram $(**)$ commutative, one has to put 
 $\tau(\nk(X))=S_{X}+\Ik$. This definition is 
  well posed  by $(*)$, and it provides a semiring homomorphism of $\Nk$ 
  into $\RR/\Ik$, by
  Proposition \ref{chars}$(2)$. Moreover $\tau$ is one-to-one and 
  onto $(\RP\!+\Ik)/\Ik$
  because of
  Proposition 
 \ref{chars}$(5)$. So $\tau$ can 
  be uniquely extended to the required ring isomorphism $\tau:\Rk\to\RR/\Ik$, by
  Proposition \ref{chars}$(6)$.

		     \qed

\section{Asymptotic numerosities and quasi-selective ultrafilters}\label{asy}

\medskip
Various measures of size for infinite sets of 
positive integers  $A\incl\N^{+}$, commonly used in number theory, are 
obtained by considering the sequence of ratios $$\frac{|\{a\in A\mid a\le 
n\}|}{n}.$$
In fact, the \emph{(upper, lower) asymptotic density} of $A$ are 
defined as the limit 
(superior, inferior) of this sequence.
This procedure might be viewed as measuring \emph{the ratio between 
the numerosities of 
$A$ and $\N^{+}$}. In this perspective, one should assume that $\nk(A)\le\nk(B)$ 
whenever the sequence of ratios for $B$ dominates that for $A$.
So one is led to introduce the following notion of 
``asymptotic'' equinumerosity 
relation.

\begin{definition}\label{dasy}
    \emph{For $X\incl\N^{k}$ put $$X_{n}=\{x\in X\mid x_{i}\le n\ 
    \mbox{for}\ i=1,\ldots,k\}.$$
    The equinumerosity relation $\eq$ is \emph{asymptotic} if:}
    
\smallskip
\begin{enumerate}
	\item[\easy] \emph{  If $|X_{n}|\le |Y_{n}|$ for all $n\in\N$,
	then there exists $Z\incl Y$ such that $X\eq Z$. }
    \end{enumerate}
    
\end{definition}
\medskip
Remark that sets $X$ and $Y$ are not assumed to be of the \emph{same 
dimension}. In fact, at the end of this section we shall use the 
condition \easy\ to give a notion of ``quasi-numerosity'' that is 
defined on \emph{all sets of tuples} of natural numbers.
 
According to this definition, if $\nk:\WW\to\Nk$ is the numerosity function 
associated to an
 equinumerosity $\,\eq$, then $\,\eq$ is asymptotic if and only if 
  $\nk$ satisfies the 
     following property for all  $X,Y\in\WW$:
     $$|X_{n}|\le |Y_{n}|\ \mbox{\rm{for all}}\ n\in\N\ \ \Imp\  
     \ \nk(X)\le \nk(Y).$$

     In the following theorem we give a nice algebraic characterization of 
asymptotic numerosities. Namely   
   \begin{itemize}
       \item Let   $\eb^{n}$ be the sequence made up of 
  $(n+1)$-many ones. For $S\in \RR$\ let
 $ S(\eb^{n})$ be the value taken by $S$ when $1$ is assigned 
  to the variables $t_{j}$ for $0\le j\le n$, while $0$ is assigned 
  to the remaining
 variables. 
 So 
 \begin{center}
     $|X_{n}|=S_{X}(\eb^{n})$ for every point set $X$.
 \end{center}   
       
\end{itemize}
   Then we have
    \begin{thm}\label{tult} Define the map $\Phi:\RR\to\Z^{\N}$ by
       $\Phi(S)= \la\, S(\eb^{n})\,\ra_{ n\in\N}$. Then

\smallskip
\noindent
$(i)$  $\Phi$ is a ring homomorphism, whose kernel $\Kk$
	    is disjoint from $\RP$.
	     The range of $\Phi$ is  the subring  $\,\Pp$ of $\,\Z^{\N}$ consisting of all
	 polynomially bounded sequences 
	 $$\Pp=\{g:\N\to\Z\mid \exists k,m\, \forall n>m\ 
	 |g(n)|<n^{k}\}\incl \Z^{\N}.$$
	
	    \noindent
	    $(ii)$ Let $\nk:\WW\to\Nk$ be the numerosity function 
	    associated to an asymptotic equinumerosity $\eq$.
	    Then there exists a unique ring homomorphism 
    $\psi$  of $\Pp$ onto the ring $\Rk$
    generated by $\Nk$ such that 
the following diagram commutes

 \medskip
	   \bigskip
	   \begin{center}
	   \begin{picture}(180,60)
	      \put(0,0){\makebox(0,0){$\Nk$}}
	      \put(90,0){\makebox(0,0){$\Rk$}}
	      %\put(90,33){\makebox(0,0){\diult}}
	      \put(180,0){\makebox(0,0){$\Pp$}}
	      \put(45,60){\makebox(0,0){$\WW$}}
	      \put(132,60){\makebox(0,0){$\RR$}}
	      \put(45,6){\makebox(0,0){$j$}}
	      %\put(135,-6){\makebox(0,0){$\cong$}}
	      \put(135,6){\makebox(0,0){$\psi$}}
	      \put(12,36){\makebox(0,0){$\nk$}}
	      %\put(81,36){\makebox(0,0){$\nk_\W\!\!
	      %\upharpoonleft\!\!_{\mathbb{W}_0}$}}
	      \put(161,36){\makebox(0,0){$\Phi$}}
	      \put(90,66){\makebox(0,0){$\Sg$}}
	      \put(12,0){\vector(1,0){64}}
	      \put(170,0){\vector(-1,0){66}}
	      \put(40,52){\vector(-1,-1){40}}
	      %\put(46,52){\vector(1,-1){40}}
	      \put(135,52){\vector(1,-1){40}}
	      \put(57,60){\vector(1,0){62}}
	   \end{picture}
	   \end{center}
	   
	   \bigskip	   	   \noindent
	   $($where $\Sg$ maps any $X\in\WW$ to its characteristic
	    series $S_{X}\in\RP,$ 
	    and $j$ is the natural 
	    embedding.$)$  
	   
	   \medskip
	    \noindent
	    $(iii)$ Under the hypotheses in $(ii)$, there exists a quasi-selective ultrafilter $\U$ on $\N$ such that
the sequence	    $g\in\Pp$ belongs to kernel of  $\psi$ if and only if its 
	    zero-set
             $Z(g)=\{n\mid g(n)=0\}$ belongs to $\U$. In particular
 	     \begin{equation*}
          X\eq Y\ \Iff\ \{n\in\N\mid |X_{n}|= |Y_{n}|\}\in\U. 
	  \tag{$\sharp$}
             \label{sh}
         \end{equation*}

    \end{thm}

    \bigskip
   \dim 
   
   $(i)$ The first  assertion is immediate. In order to characterize 
   the range of $\Phi$, recall that
   there are $n^{k}$ 
   (noncommutative) monomials of degree $k$ in $n$ variables. On 
   the one hand,  if $\deg 
   S<d$ and all the coefficients of $S$ are bounded by $B$, then
   $|S(\eb^{n})|< Bn^{d}$. On the other 
   hand, assume that $|g(n)-g(n-1)|<n^{k}$ for $n>m$, say. Pick a 
   (homogeneous)
   polynomial  $p(t_{0},\ldots,t_{m})$ such that  $p(\eb^{0})=g(0),\ldots,
   p(\eb^{m})=g(m)$. Then, for each $n>m$, add (or subtract) exactly 
   $|g(n)-g(n-1)|$ monomials of degree not exceeding $k$ in the variables 
   $t_{0},\ldots,t_{n}$,  where $t_{n}$ actually appears. 
   The resulting series $S\in\RR$ clearly verifies $S(\eb^{n})=g(n)$.
   Notice that if all monomials are chosen of the same degree, then $S$ is 
   the difference of two characteristic series. Hence any $S\in\RR$ 
   is congruent modulo $\Kk$ to a difference $S_{X}-S_{Y}$.
   
   \smallskip
   $(ii)$ The gauge ideal $\Ik$ corresponding to $\eq$ contains $\Kk$ 
   because the equinumerosity $\eq$ is asymptotic. Therefore 
   $\Rk\cong\RR/\Ik\cong(\RR/\Kk)/(\Ik/\Kk)\cong\Pp/\Phi[\Ik]$. So, 
   in the notation of Theorem \ref{ser}, 
   $\psi$ is the unique ring homomorphism such that 
   $\tau\circ\psi\circ\Phi=\pi$.
   
    \smallskip
   $(iii)$ The kernel  $\ker\psi=\II$  is a prime ideal of $\Pp$, because 
   $\Rk$ is a domain. The idempotents of $\Pp$ are all and only the
   characteristic functions $\chi_{A}$ of subsets 
   $A\incl\N$. Hence, for every $A\incl\N$, the ideal $\II$ contains 
   exactly one of the two  complementary 
   idempotents $\chi_{A},1-\chi_{ A}=\chi_{\N\/ A}$. Moreover 
   $1-\chi_{A\cap 
   B}=1-\chi_{A}\chi_{B}=1-\chi_{A}+\chi_{A}(1-\chi_{B})$,
   and
   $A\incl B$ implies 
   $1-\chi_{ B}=(1-\chi_{A})(1-\chi_{B})$. Hence the set 
   $$\U=\{A\incl\N\,\mid\, 
   1-\chi_{A}\in\II\,\}$$ is an ultrafilter on $\N$.
   
   Now $g=g(1-\chi_{Z(g)})$, so $g$ belongs to $\II$ whenever 
   $Z(g)\in\U$. Conversely, if 
   $Z(g)\not\in\U$, then $\psi(\chi_{Z(g)})=0$ and 
   $\psi(g^{2}+\chi_{Z(g)}-1)$ is non-negative. It follows that 
   $\psi(g)\ne 0$.
  In particular
   $$X\eq Y\ \Iff\ \Phi(S_{X}-S_{Y})\in\ker\psi\ \ \Iff\
\{n\in\N\mid |X_{n}|- |Y_{n}|=0\}\in\U$$

   It remains to show that $\U$ is 
quasi-selective. To this aim, remark that, modulo 
the gauge ideal $\Ik=\ker(\psi\circ\Phi)=\Phi^{-1}[\II]$,
each series in $\RR$ is equivalent 
 to $\,\pm P$ 
for some $P\in\RR^{+}$. In turn such a $P$ is mapped by 
$\Phi$ to a nondecreasing element of $\Pp$. As $\Phi$ is surjective, 
we may conclude that every polynomially bounded non-negative sequence 
is $\U$-equivalent to a nondecreasing one. Hence $\U$ is 
quasi-selective.
   
   \qed 
 
   \bigskip
 This construction allows for classifying all asymptotic 
 equinumerosity relations
 by means of quasi-selective ultrafilters on $\N$ as follows:

\begin{cor}\label{tasy}
There exists a biunique correspondence between asymptotic 
equinumerosity relations on the
space of point sets over $\N$ and quasi-selective ultrafilters on $\N$. 
In 
this correspondence, if the equinumerosity $\eq$ corresponds to the 
ultrafilter  $\U$,  then
 \begin{equation*}
          X\eq Y\ \Iff\ \{n\in\N\mid |X_{n}|= |Y_{n}|\}\in\U. 
	  \tag{$\sharp$}
             \label{sh}
         \end{equation*}
More precisely, let $\Nk$ be the set of numerosities of $\eq$, 
and let 
$\nk$ be the corresponding numerosity function.
Then the map $$\fg:\nk(X)\mapsto\, [\la\, |X_{n}|\,\ra_{{n\in\N}}]_{\U} $$
is an isomorphism of ordered semirings between $\Nk$ 
and  the initial segment $\Pp_{\U}$ of the
ultrapower $\N\ult{\N}{\U}$ that
contains the classes of all polynomially bounded sequences.
In particular, asymptotic numerosities are, up to isomorphism,
nonstandard integers.

\end{cor}

\proof
Given an asymptotic equinumerosity $\eq$, let $\Phi,\psi$, and 
the ultrafilter $\U=\U_{\eq}$ be as in Theorem \ref{tult}. Then
$\U$ is quasi-selective, and 
$$X\eq Y\ \Iff\ \Phi(S_{X}-S_{Y})\in\ker\psi\ \ \Iff\
\{n\in\N\mid |X_{n}|- |Y_{n}|=0\}\in\U,$$
that is (\ref{sh}).

Conversely, given a quasi-selective ultrafilter $\U$ on $\N$, define 
the equivalence $\eq_{\U}$ on $\WW$ by (\ref{sh}). Then \euno,\etre, and 
\ecinque\ are immediate. We prove now the following strong form of \easy\ 
that implies also 
\edue: 
$$\{n\in\N\mid 
|X_{n}|\le|Y_{n}|\}\in\U\ \ \Imp\ \ \exists Z\incl Y\ 
\mbox{s.t.}\ Z\eq_{\U}X.$$
Assume that  $\{n\in\N\mid 
|X_{n}|\le|Y_{n}|\}=U\in\U$: by quasi-selectivity there exists 
$V\in\U$ such that $|X_{n}|$, $|Y_{n}|$, and $|Y_{n}|-|X_{n}|$ are
nondecreasing on $V$. Then one can isolate from $Y$ a subset $Z\eq 
X$ in the following way: if $m,m'$ are consecutive elements of $U\cap V$,  put in $Z$ exactly 
$|X_{m'}|-|X_{m}|$ elements of $Y_{m'}\/ Y_{m}$. 

So $\nk(X)<\nk(Y)$ if and only if $\fg(\nk(X))<\fg(\nk(Y))$ in 
$\Pp_{\U}$, and the last assertion of the theorem follows. Finally, 
remark that every subset $A\incl\N$ can be written as 
$A=\{n\in\N\mid |X_{n}|= |Y_{n}|\}$ for suitable point sets $X,Y$. So 
the biconditional (\ref{sh}) uniquely determines the ultrafilter 
$\U=\U_{\eq}$, and one has
\begin{itemize}
    \item  $\U_{(\eq_{\U})}=\,\U$ for all quasi-selective ultrafilters 
    $\U$;

    \item $\eq_{(\U_{\eq})}=\,\eq$ for 
all asymptotic equinumerosity relations $\eq$. 
\end{itemize}
 Hence the 
correspondence between $\eq$ and $\U$ is biunique.
 
\qed

\medskip
It is worth remarking that the property \easy\  yields at once  \etre\ 
and
commutativity of product,
as well as many other 
natural instantiations of the fourth Euclidean common notion:
\begin{quote}
    \emph{``Things applying $[$exactly$]$ to one another are equal 
    to one another.''}
\end{quote}
More precisely, if the \emph{support} of a tuple is defined by $$\mbox{supp}\, 
(x_{1},\ldots,x_{k})=\{x_{1},\ldots,x_{k}\},$$ then 
all \emph{support preserving bijections} can be taken as 
``congruences'' for asymptotic numerosities, because any such 
$\sg:X\to Y$ maps $X_{n}$ onto $Y_{n}$ for all $n\in\N$.

Actually, in order to give a ``Cantorian'' characterization of asymptotic 
equinumerosities, we isolate a wider class of bijections, namely
\begin{definition}
    Let $\U$ be a nonprincipal ultrafilter on $\N$. A bijection $\sg:X\to Y$ 
    is a
    $\U$-congruence if $\{n\in\N\mid \sg[X_{n}]=Y_{n}\}\in\U$.
\end{definition}
When the ultrafilter $\U$ is quasi-selective, the $\U$-congruences
determine an asymptotic equinumerosity:
\begin{cor}
    Let $\eq$ be an asymptotic equinumerosity, and let $\U$ be the 
    corresponding quasi-selective ultrafilter. Then $X\eq Y$ if and 
    only if there exists a $\U$-congruence $\sg:X\to Y$.
    \qed
    
\end{cor}

This point of view allows for an interesting generalization of the 
notion of asymptotic equinumerosity to \emph{all subsets} of 
$\EE=\bigcup_{k\in\N^{+}}\N^{k}$, namely
\begin{itemize}
    \item  put $\EE_{n}=\bigcup_{1\le k\le n}\{0,\ldots,n\}^{k}$; 

    \smallskip
    \item   let 
    $\U$ be a filter on $\N$;
    
\smallskip
    \item  call a map $\sg:\EE\to\EE$  a $\U$-\emph{isometry} if
    $\{n\mid\sg[\EE_{n}]=\EE_{n}\}\in\U;$  

    \smallskip
    \item let
	$\Sk_{\U}$ be the group of all $\U$-isometries, and 
     for $X,Y\incl\EE$ put $$X\eq_{\U} Y\ \  \Iff\ \ \mbox{there 
     exists}\  
    \sg\in\Sk_{\U}\ \mbox{such that}\ \sg[X]=Y.$$
    
\end{itemize}
If $\U$ is a quasi-selective ultrafilter we can now assign a 
``quasi-numerosity'' to every subset  of $\EE$,  namely its orbit 
under $\Sk_{\U}$:
$$\nku(X)=[X]_{\equ}=X^{\Sk_{\U}}=\{\sg[X]\mid \sg\in\Sk_{\U}\}.$$
Let $\Nku=\P(\EE)/\!\!\equ$ be the set of quasi-numerosities, and 
let $\SSU$ be the initial segment of the
    ultrapower $\N\ult{\N}{\U}$ determined by $\nk_{\U}(\EE)=[\la 
(n+1)\frac{(n+1)^{n}-1}{n}\mid n\in\N\ra]_{\U}$. 

Notice that if $X\incl\N^{k}$, then $X_{n}=X\cap \EE_{n}$ for $n\ge 
k$. So we have

\begin{thm}\label{tperm}
    The relation  $\eq_{\U}$ is an equivalence on $\P(\EE)$ that satisfies 
    the properties 
    \emph{\easy,\euno},\emph{\edue}, and, when restricted to $\WW$, agrees with 
    the asymptotic equinumerosity corresponding to $\U$.
    Moreover the map 
    $$\fg_{\U}:\nku(X)\mapsto\, [\la\, |X\cap \EE_{n}|\,\ra_{{n\in\N}}]_{\U} $$
preserves sums and is an order isomorphism  between $\Nku$ 
    and   $\,\SSU\incl \N\ult{\N}{\U}$.
\qed
\end{thm}

  It is worth mentioning that both the ``multiplicative'' 
properties 
\etre\ and \ecinque\  can fail for sets of 
infinite dimension. \Pes $$\{(0,1)\}\*\EE\pincl \{0\}\*\EE\pincl \EE$$ 
have increasing quasi-numerosities, thus contradicting both 
\etre\ and \ecinque.
%%%AB: changed "pseudo-numerosities" to "quasi-numerosities" 
%%%OK

\section{Final remarks and open questions}\label{froq}

It is interesting to remark that the non-selective quasi-selective 
ultrafilter $\V$ defined in the proof of Theorem \ref{constr} satisfies the 
following ``weakly Ramsey'' property:\footnote{~
Ultrafilters satisfying this property have been introduced in 
\cite{bl74} under the name \emph{weakly Ramsey}, and then generalized 
to $(n+1)$-Ramsey ultrafilters in \cite{ros}.}
\begin{center}
   \emph{for any finite coloring of $[\N]^{2}$ there is $U\in\V$ 
   such that $[U]^{2}$ has only two colors.}    
\end{center}

It is easily seen that if $\V$ is weakly Ramsey, then every function is either 
one-to-one or nondecreasing modulo $\V$.
So both weakly Ramsey and quasi-selective  ultrafilters are  
P-points. However  
these two classes are different whenever there exists a non-selective 
P-point, because the former is closed under isomorphism, whereas the 
latter is not, by Proposition \ref{blp}.

In \zfc, one can draw the following diagram of implications
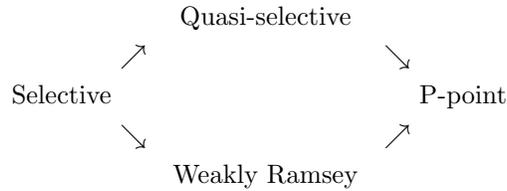
\begin{eqnarray*}
    {} & \mbox{Quasi-selective} & {}  \\
    \nearrow & {} &\searrow \\
   \mbox{Selective}\ \ \ \   &{}  & \ \ \ \  \mbox{P-point}  \\
   \searrow & {} & \nearrow \\
    {} & \mbox{Weakly Ramsey} & {}
\end{eqnarray*}

 Recall that, assuming \CH, the following facts hold:
\begin{itemize}
    \item  there exist  quasi-selective weakly Ramsey 
ultrafilters that are not selective (Theorem \ref{constr}); 

    \item the class of quasi-selective non-selective ultrafilters is 
not closed under isomorphisms (Proposition \ref{blp});

    \item  there are non-weakly-Ramsey P-points (see Theorem 2 of 
    \cite{bl74}).
\end{itemize}

It follows that, in the diagram above, no arrow can be reversed nor inserted 

The relationships between quasi-selective 
and weakly Ramsey ultrafilters are extensively studied in \cite{qswr}.
In particular, it is proved there that both quasi-selective and weakly 
Ramsey ultrafilters are  
P-points of a special kind, since they share the property that 
every function is equivalent to an
\emph{interval-to-one} function. So the question naturally arises as to 
whether this class of ``interval P-points'' is distinct from either 
one
of the other three classes.

Many weaker conditions than the Continuum Hypothesis have been 
considered in the literature, in order to get more information 
about special classes of 
ultrafilters on $\N$. Of particular interest are (in)equalities among 
the so called ``combinatorial cardinal characteristics of the
continuum''. 
(\Pes\ one has that P-points or selective 
ultrafilters are generic if $\ck=\dk$ or $\ck=\cov(\B)$, respectively..
Moreover if $\cov(\B)<\dk=\ck$ then there are filters that are included 
in P-points, but cannot be 
extended to selective ultrafilters. See the 
comprehensive survey \cite{bl1}.)
We conjecture that similar hypotheses can settle the problems 
mentioned above. 
 
\bigskip

	    It is worth mentioning that, given a quasi-selective 
	    ultrafilter $\U$, the corresponding asymptotic ``quasi-numerosity'' 
	    $\nk_{\U}$ of 
	    Theorem \ref{tperm} can be extended to all subsets of 
	    the \emph{algebraic Euclidean space} 
	    $$\QQ=\bigcup_{k\in\N^{+}} \overline{\Q}^{\,k}, \ \ 
	    \mbox{where}\ 
	    \overline{\Q}\ \mbox{is the field of all algebraic numbers.}$$
To this end,  replace the sequence of finite sets $\EE_{n}$ by
$$\QQ_{n}=\{(\ag_{1},\ldots,\ag_{k})\in\QQ\,\mid\,k\le n,\ \mbox{and}\
\exists  a_{ih}\in\Z,\, |a_{ih}|\le n, 
\sum_{0\le h\le n} a_{ih}\ag_{i}^{h}=0\,
\}.$$ 
Then extend the definition of $\U$-\emph{isometry} to maps 
$\sg:\QQ\to\QQ$ such that
    $$\{n\mid\sg[\QQ_{n}]=\QQ_{n}\}\in\U,$$
and, for $X,Y\incl\QQ$ put
$$X\eq_{\U} Y\ \  \Iff\ \ \mbox{there 
     exists a $\U$-isometry  
    $\sg$ such that}\ \sg[X]=Y.$$

    Remark that the sequence $\la\, 
    |\QQ_{n}|\,\ra_{{n\in\N}}$ belongs 
    to $F_{\U}$, since it is bounded by $\la\, 
    n^{3n^{2}}\,\ra_{{n\in\N}}\,$, 
    say. Hence 
one obtains the following natural extension of Theorem \ref{tperm}.

\begin{thm}\label{alg}
    The relation  $\eq_{\U}$ is an equivalence on $\P(\QQ)$ that satisfies 
    the properties 
    \emph{\easy,\euno},\emph{\edue}, and, when restricted to 
    $\bigcup_{k\in\N}\P(\overline{\Q}^{k})$, also \emph{\etre}\ 
    and \emph{\ecinque}.\\
    The map $$\nku(X)\mapsto\, 
    [\la\, |X\cap \QQ_{n}|\,\ra_{{n\in\N}}]_{\U} $$
    preserves sums and is an order isomorphism  between the set of 
   ``asymptotic quasi-numerosities''
   $\Nku=\P(\QQ)/\!\!\eq_{\U}$ 
    and  the initial segment $\,\TTU\incl \N\ult{\N}{\U}$ determined 
    by $\nk_{\U}(\QQ)=[\la\, |\QQ_{n}|\,\ra_{{n\in\N}}]_{\U}.$
\qed
\end{thm}

	Similar results hold for point sets over any \emph{countable} line 
	$\L$ equipped 
with a \emph{height function} $h$, 
provided that the corresponding 
function $g:\N\to\N$ defined by $g(n)=|\{x\in\L\mid h(x)\le n\,\}|$ belongs to 
the class $F_{\U}$ of Section \ref{qsu}. If this is not the case, 
one can still maintain a biunique correspondence between  
asymptotic equinumerosities and
	     ultrafilters, by restricting to 
	      ultrafilters $\U$ with the 
	     property that 
	    every function bounded by $g$ is $\U$-equivalent to a 
	    nondecreasing one (\emph{$g$-quasi-selective} ultrafilters).

The question as to whether  there exist equinumerosities which are 
not asymptotic with respect to suitable height functions is still open. 
Of particular interest 
might be the identification of equinumerosities whose existence 
is provable in \zfc\ alone. Actually, the very notion of  gauge ideal 
has been introduced in order to facilitate the investigation of these 
most general 
equinumerosity relations.

\bigskip
\bibliographystyle{amsplain}

\end{document}